\documentclass[12pt]{amsart}
\usepackage[english]{babel}
\usepackage[T1]{fontenc}
\usepackage[utf8]{inputenc}
\usepackage{amsmath,amssymb,amsthm}
\usepackage{textcomp}
\usepackage{mathrsfs}
\usepackage{stmaryrd}
\usepackage[scr=boondoxo]{mathalfa}
\usepackage[numbers]{natbib}
\usepackage{mathtools}

\usepackage[all]{xy}
\usepackage{hyperref}
\usepackage{tikz}
\usepackage{quiver}

\newcommand{\C}{\mathbb{C}}

\newcommand{\N}{\mathbb{N}}

\newcommand{\R}{\mathbb{R}}

\newcommand{\T}{\mathbb{T}}
\newcommand{\Z}{\mathbb{Z}}

\newcommand{\Bcal}{\mathcal{B}}
\newcommand{\Hcal}{\mathcal{H}}
\newcommand{\Ical}{\mathcal{I}}

\newcommand{\Lcal}{\mathcal{L}}
\newcommand{\Mcal}{\mathcal{M}}
\newcommand{\Scal}{\mathcal{S}}
\newcommand{\gscr}{\mathscr{g}}
\newcommand{\hscr}{\mathscr{h}}

\newcommand{\lscr}{\mathscr{l}}

\newcommand{\pscr}{\mathscr{p}}

\newcommand{\Escr}{\mathscr{E}}
\newcommand{\Sscr}{\mathscr{S}}
\newcommand{\Tscr}{\mathscr{T}}
\newcommand{\Uscr}{\mathscr{U}}

\newcommand{\Xscr}{\mathscr{X}}

\newcommand{\Afrak}{\mathfrak{A}}

\newcommand{\Gfrak}{\mathfrak{G}}

\newcommand{\Sfrak}{\mathfrak{S}}

\DeclareMathOperator{\Vect}{\mathrm{Vect}}
\DeclareMathOperator{\Hom}{\mathrm{Hom}}
\DeclareMathOperator{\coker}{\mathrm{coker}}
\DeclareMathOperator{\GL}{\mathrm{GL}}
\DeclareMathOperator{\Cone}{\mathrm{Cone}}
\DeclareMathOperator{\Spec}{\mathrm{Spec}}

\numberwithin{equation}{section}

\begin{document}
\title{Non-simplicial quantum toric varieties}
\author{Antoine Boivin}

\address{Antoine Boivin,Univ Angers, CNRS, LAREMA, SFR MATHSTIC, F-49000 Angers, France}%
\email{\href{mailto:antoine.boivin@univ-angers.fr}{antoine.boivin@univ-angers.fr}}
\subjclass{53D20 (Primary) 81S10, 53D37 (Secondary)}
\keywords{Toric geometry, irrational fans, non-simplicial fans, stacks}



\newtheorem{Thm}{Theorem}[subsubsection]
\newtheorem{Prop}[Thm]{Proposition}
\newtheorem{Lemma}[Thm]{Lemma}
\newtheorem{Cor}[Thm]{Corollary}
\newtheorem{Const}[Thm]{Construction}
\newtheorem{War}[Thm]{Warning}
\theoremstyle{definition}

\newtheorem{Ex}[Thm]{Example}

\newtheorem{Def}[Thm]{Definition}
\newtheorem{Not}[Thm]{Notation}
\theoremstyle{remark}
\newtheorem{Rem}[Thm]{Remark}

\begin{abstract}
This paper defines quantum toric varieties associated with an arbitrary fan in a finitely generated subgroup of some $\R^d$.  This is a generalization of the results of the article \citep{VQS} of Katzarkov, Lupercio, Meersseman and Verjovsky.
\end{abstract}

\maketitle

\tableofcontents

\renewcommand{\theThm}{\arabic{section}.\arabic{Thm}}
\section{Introduction}

A toric variety is a complex algebraic variety with an action of an algebraic torus $(\C^*)^n$ with a Zariski open orbit isomorphic to this torus. Such a variety can be described by a fan of rational strongly convex polyhedral cones on a lattice $\Gamma$ (see, for example, \citep{cox} or \citep{fulton}). 

More precisely, there is an equivalence between the category of toric varieties and that of fans. This central result gives us a dictionary between toric varieties' geometric properties and fans' combinatorial properties making toric varieties one of the most studied classes of complex algebraic varieties.

However, in the classical theory, this fan has to be rational. Therefore the toric varieties are rigid, i.e., we cannot deform them. Indeed, if we deform a lattice, it can become non-discrete (for instance, the group $\Z + \alpha\Z$ is discrete if $\alpha$ is rational but is dense in $\R$ if $\alpha$ is irrational). Hence, the natural question to ask is: What can we do if we do not have a fan living on a lattice but rather on a finitely generated (possibly irrational) subgroup of some $\R^d$ (also called quasi-lattice in reference of its use in quasi-crystals theory, see \citep{senechal1996quasicrystals})? 
We can find in the literature several constructions of such generalized toric varieties: as leaf space (see \cite{LV} for the classical (projective) construction, \cite{toricfoliation}, \cite{Battaglia_2015} and \cite{katzarkov:hal-01672716} for the non-rational one), as (possibly stratified) symplectic quasifold global quotient (see \cite{quasifold} for the simplicial polytopal case and \cite{battaglia2008geometric} for the general polytopal case) or as stacks (see \cite{hoffman2020toric} for the polytopal case as compact symplectic stacks and \cite{VQS} for the simplicial case). In this paper, we will use the framework of the last cited article.

The authors of \citep{VQS} gave a construction of "quantum toric varieties" (which are stacks) described by a simplicial fan (i.e., the 1-cones contained on each cone of the fan are $\R$-linearly independent) on a finitely generated subgroup $\Gamma$ of some $\R^d$. When the fan is rational, one recovers the classical toric variety, but irrational simplicial fans is also covered. 
This construction is functorial and defines an equivalence between the quantum toric varieties category and that of quantum simplicial fans (see theorems 5.18 and 6.24 of \citep{VQS}). Hence, we have a practical notion of morphisms between these quantum toric varieties. Moreover, this construction behaves well with non-effective quotients which can appear when we fix the number of generators of the group $\Gamma$ to compute moduli spaces of quantum toric varieties.

Now, the simplicial fans form only a tiny part of all the fans: in the classical theory, they correspond to the toric varieties which are orbifold (i.e. with cyclic singularities). Hence, the restriction to simplicial fans is a strong one.

This paper extends the construction of \citep{VQS} to the general case, i.e., we define the quantum toric variety associated with an arbitrary fan. 

The non-simplicial case brings new problems. Indeed, in the simplicial case, the family of 1-cones of a cone is $\R$-linearly free and can be completed to a basis of $\R^d$. Hence, up to isomorphism, a simplicial cone is a standard cone $\Cone(e_1,\ldots,e_k) \subset \R^d$ (where $\{e_1,\ldots,e_d\}$ is the canonical basis of $\R^d$). Moreover, we can easily find the faces of these cones: they are the cones generated by a subfamily of $\{e_1,\ldots,e_k\}$. This basic fact is repeatedly used in construction of \citep{VQS}: this is no longer possible for the non-simplicial case. Indeed, as the 1-cones of the non-simplicial cones are linearly dependant, we do not have a notion of standard cone and the cones generated by a subfamily of the 1-cones are not necessarily a face of the cone since the cone can have an arbitrarily large number of 1-cones. The main consequences of this irregularity of the number of generators are the absence of simple local models for quantum toric varieties, the lack of morphisms for non-calibrated quantum toric varieties and the absence of realization of the non-simplicial quantum toric varieties as a global quotient stack defined by the action induced by Gale transform of the generators (in the same way as \cite{Coxquotient} for the classical toric varieties and as the toric Deligne-Mumford stack case of \citep{jiang2008} or the quantum GIT of the simplicial case of \citep{VQS})

One of the key technical points is to replace the calibration $h$ of the group $\Gamma$ used in \citep{VQS} by a calibration $\varphi$ of a group isomorphic to $\Gamma$ but which is a subgroup of a higher dimensional $\R^p$ for each maximal cone of the fan ($p$ is the number of 1-cones of the maximal cone).

Section \ref{Quantum_tori} gives a suitable definition of quantum tori (called "presented quantum tori") to deal with the non-simplicial case. More precisely, these quantum tori encode the last paragraph's calibration change.  

In section \ref{NSQTV}, we define the affine quantum toric variety (which will be calibrated since we have to keep track of the relations between the 1-cones) associated with a (non-simplicial) cone and more generally, the quantum toric variety with the descent data of the affine pieces associated to a quantum fan.
In particular, this construction works for simplicial fans and coincide with that of \citep{VQS}. At the end of this section, we prove the main theorem of this paper (extending the similar theorem of \citep{VQS}): 
\begin{Thm}
The category of calibrated quantum fans and the category of (calibrated) quantum toric varieties are equivalent.
\end{Thm}
Finally, we realize these quantum toric varieties as a global stack quotient 
in section \ref{GIT} and prove that they can not be done in the same way as it was done in the simplicial case.


After that, we realize quantum toric varieties as a gerbe over a "non-calibrated quantum toric varieties" associated to it (i.e. where we replace the action of $\Z ^{N-d}$ through the epimorphism $\varphi$ by the action of $\varphi(\Z^{N-d})$ and hence removing the ineffective part of the action) in section \ref{forget_cal}. However, in the non-simplicial case, we cannot define non-calibrated quantum toric varieties and their morphisms alone as in the simplicial case. Indeed, the presence of relations between the generators of the non-simplicial cones prevents the construction of morphisms from the morphisms of non-calibrated quantum fans.


\renewcommand{\theThm}{\arabic{section}.\arabic{subsection}.\arabic{Thm}}
\section{Notations and conventions}
\subsection{Stacks}
We will take the same conventions on stacks as \citep{VQS} (see \citep{intro_stack}, \citep{Stackeveryb}, \citep{stacks-project} for details on stacks) :

Let $\Afrak$ be the category of affine toric varieties with toric morphisms. We endow this category with a structure of site by the consideration of the coverings $\left\{ U_i \hookrightarrow T\right\}_{i \in \{1,\ldots,n\}} $, where the $U_i$ are toric open subsets of $T$ and $T=\bigcup_{i=1}^n U_i$. 

Let $\Gfrak$ be the category of complex analytic spaces carrying an action of an abelian complex Lie group $G$ with an open orbit isomorphic to $G$.
\begin{Def}

Let $H$ be an abelian Lie group and $X$ be an object of $\Gfrak$ with an action of $H$ commute with the action of $G$. \\
The stack $[X/H]$ is the stack over $\Afrak$ whose objects over $T$ are $H$-principal bundle $\widetilde{T} \to T$ (in $\Gfrak$) with an $H$-equivariant morphism $\widetilde{T} \to X$ and morphisms over $S \to T$ are a bundle morphism $\widetilde{S} \to \widetilde{T}$ compatible with the equivariant maps.  
\end{Def}

All the stacks in this paper will be of this form or given by the descent data of stacks of this form.

\subsection{Fans}
We will recall some definitions on (calibrated) quantum fans: 

\begin{Def}
Let $\Gamma$ be a finitely generated subgroup of $\R^d$ such that $\Vect_\R(\Gamma)=\R^d$. A calibration of $\Gamma$ is given by:
\begin{itemize}
\item A group epimorphism $h \colon \Z^N \to \Gamma$
\item A subset $\Ical \subset \{1,\ldots,N\}$ such that $\Vect_\C(h(e_j), j \notin \Ical)=\C^d$ (this is the set of virtual generators)

\end{itemize}
This is a standard calibration if $\Z^d \subset \Gamma$, $h(e_i)=e_i$ for $i=1,\ldots,d$ and $\Ical$ is of the form $\{n-|\Ical|+1,\ldots,n\}$
\end{Def}

\begin{Def} \label{calib_q_fan_def}
A calibrated quantum fan $(\Delta,h,\Ical)$ in $\Gamma$ is the data of 
\begin{itemize}
\item a collection $\Delta$ of strongly convex polyhedral cones generated by elements of $\Gamma$ such that every intersection of cones of $\Delta$ is a cone of $\Delta$, every face of a cone of $\Delta$ is a cone and $\{0\}$ is a cone of $\Delta$.
\item a standard calibration $h$ with $\Ical$ its set of virtual generators
\item A set of generators $A$ i.e. a subset of $\{1,\ldots,N\} \setminus \Ical$ such that the 1-cone generated by the $h(e_i)$ for $i \in A$ are exactly the 1-cones of $\Delta$ 
\end{itemize}  
\end{Def}

\begin{Rem}
In the rational case, this definition corresponds to an extended stacky fan (see \citep{jiang2008})
\end{Rem}
 
Let $h: \Z^N \to \Gamma$ a calibration (where $\Ical$ is its set of virtual generators) and  $(\Delta,h,\Ical)$ a calibrated quantum fan in $\Gamma$. Note $\Delta_h$ the fan in $\Z^N$ such that 
$\sigma \in \Delta_h$ if, and only if, there exists $\Cone(h(e_i), i \in I) \in \Delta$ such that $\sigma \preceq \Cone(e_i, i \in I) $

\begin{Def}
A morphism of calibrated quantum fans between $(\Delta,h,\Ical)$ in $\Gamma$ and $(\Delta',h',\Ical')$ in $\Gamma'$ is a pair of linear morphisms $(L: \R^d \to \R^{d'},H: \R^N \to \R^{N'})$ where 
\begin{itemize}
\item 
The diagram 
\begin{equation} \label{qf_diag}
  \begin{tikzcd}
	{\R^n} & {\R^d} \\
	{\R^{n'}} & {\R^{d'}}
	\arrow["{h_\R}", from=1-1, to=1-2]
	\arrow["L", from=1-2, to=2-2]
	\arrow["H"', from=1-1, to=2-1]
	\arrow["{h'_\R}"', from=2-1, to=2-2]
\end{tikzcd}
\end{equation}

commutes (where $h_\R$ (resp. $h'_\R$) is the $\R$-linear map associated to $h$ (resp. to $h'$))
\item $L(\Gamma) \subset \Gamma'$
\item If $\sigma \in \Delta_h$ then it exists $\sigma'\in \Delta'_{h'} $ such that $H(\sigma) \subset \sigma'$ (thanks to the first point, the same statement is true for $L$)
\item If $H(e_i) \in \Cone(e_j,j \in I') \in \Delta_h$ then $H(e_i) \in\bigoplus_{j \in I'} \N e_j$  (same statement is true for $L$ and the vectors $h(e_i)$, thanks to the first point)
\item  For all $i \notin \Ical$, $H(e_i) \in \bigoplus_{j \notin \Ical'} \Z e_j$
\item There exists a map $s: \Ical \to \Ical'$ such that for all $i \in\Ical, H(e_i)=e_{s(i)}$ 
\end{itemize}

\end{Def}

\begin{Rem} ~ 

\begin{itemize}
    \item This is a global definition of quantum fan morphisms by opposition to the definition with a family of compatible morphisms between each cone.
    \item We can define non-calibrated counterpart of these definitions by replacing the data of the calibration by the data of the 1-cones $h(e_i)$, $i \in \{1,\ldots,N\}$.
\end{itemize}

\end{Rem}
\renewcommand{\theThm}{\arabic{section}.\arabic{subsection}.\arabic{Thm}}
\section{Quantum tori}
\label{Quantum_tori}

\subsection{Reminder on quantum tori}

In this subsection, we will present some definitions on the quantum tori (see the section 3 of \citep{VQS}  for the details). \\ 
A quantum torus is the replacement of the algebraic torus $\T^d \coloneqq (\C^*)^d\simeq \C^d/\Z^d$ in the irrational setting:

\begin{Def}
Let $\Gamma$ be a finitely generated subgroup of $\R^d$. The quantum torus associated to $\Gamma$ is the quotient stack
\[
\Tscr_{\Gamma} \coloneqq [\C^d/\Gamma]
\]
where $\Gamma$ acts on $\C^d$ by translation.
\end{Def}

Since we want to keep track of the generators thanks to the calibration $h \colon \Z^N \to \Gamma$, we have to consider "calibrated" quantum tori: 

\begin{Def}
Let $(h \colon \Z^N \to \Gamma, \Ical)$ be a calibration of a finitely generated subgroup $\Gamma$ of $\R^d$. Then, the calibrated quantum torus associated to $(h,\Ical)$ is
\[
\Tscr^{cal}_{h,\Ical}=[\C^d/\Z^N]
\]
where $\Z^N$ acts on $\C^d$ through $h$
\end{Def}
\begin{Rem}
 In order to define morphisms between calibrated quantum toric varieties, we have to consider the pair $(\Tscr_{h,\Ical},h)$ in place of just the quantum torus $\Tscr_{h,\Ical}$.
\end{Rem}

\begin{Prop}[Multiplicative form]
Let $\Gamma$ be a finitely generated subgroup of $\R^d$ containing $\Z^d$ and $(h \colon \Z^N \to \Gamma, \Ical)$ be a calibration of $\Gamma$. Then,
\begin{itemize}
    \item we have a stack isomorphism $\Escr : \Tscr_\Gamma \to [\T^d/E(\Gamma)]$ such that the following diagram commutes
    \begin{equation} \label{exp_quotient_noncal}
    \begin{tikzcd}
	{\C^d} & {\T^d} \\
	{\Tscr_{\Gamma}} & {[\T^d/E(\Gamma)]}
	\arrow["E", from=1-1, to=1-2]
	\arrow[from=1-2, to=2-2]
	\arrow[from=1-1, to=2-1]
	\arrow["\mathscr{E}",from=2-1, to=2-2]
\end{tikzcd}\end{equation}
where \begin{equation} \label{exp_map}
    E: (z_1,\ldots,z_p) \in \C^d \mapsto (\mathrm{e}^{2i\pi z_1},\ldots,\mathrm{e}^{2i\pi z_p}) \in \T^d
\end{equation} 
and $E(\Gamma)$ acts on $\T^d$ by component-wise multiplication.
\item If we suppose $h(e_i)=e_i$ for $1 \leq i \leq d$ then we have a stack isomorphism $\Escr$ such that the following diagram commutes
\begin{equation} \label{exp_quotient_cal}
\begin{tikzcd}
	{\C^d} & {\T^d} \\
	{\Tscr^{cal}_{h,\Ical}} & {[\T^d/\Z^{N-d}]}
	\arrow["E", from=1-1, to=1-2]
	\arrow[from=1-2, to=2-2]
	\arrow[from=1-1, to=2-1]
	\arrow["\mathscr{E}",from=2-1, to=2-2]
\end{tikzcd}\end{equation}
\end{itemize}
\end{Prop}

\begin{Lemma}
Let $\Gamma$ be a finitely generated subgroup of $\R^d$ containing $\Z^d$ and $(h \colon \Z^N \to \Gamma, \Ical)$ be a calibration of $\Gamma$. Then,
\[\Tscr^{cal}_{h,\Ical} \simeq \Tscr_{\Gamma} \times \mathscr{B}\ker(h)\]
where $\mathscr{B}\ker(h)=[\Spec\C/\ker(h)]$.
\end{Lemma}

\begin{Def} \label{torus_morphis_noncal}
A (non-calibrated) torus morphism $\Tscr_{\Gamma} \to \Tscr_{\Gamma'}$ is a stack morphism $\lscr : \Tscr_{\Gamma} \to \Tscr_{\Gamma}$ such that there exists a linear morphism $L \colon \R^d \to \R^{d'}$ which sends $\Gamma$ on $\Gamma'$ and such that the diagram
\begin{equation*} 
  \begin{tikzcd}[ampersand replacement=\&]
	{\C^d} \& {\C^{d'}} \\
	{\Tscr_{\Gamma}} \& {\Tscr_{\Gamma'}}
	\arrow["{L_\C}", from=1-1, to=1-2]
	\arrow[from=1-1, to=2-1]
	\arrow[from=1-2, to=2-2]
	\arrow["\lscr"', from=2-1, to=2-2]
\end{tikzcd}
\end{equation*}

commutes
\end{Def}

\begin{Def} \label{torus_morphis_cal}
A (calibrated) torus morphism $\Tscr^{cal}_{h,\Ical} \to \Tscr^{cal}_{h',\Ical'}$ is a stack morphism $\lscr^{cal} : \Tscr_{h,\Ical} \to \Tscr_{h',\Ical'}$ such that there exists a linear morphism $L \colon \R^d \to \R^{d'}$, a linear morphism $\Z^n \to \Z^{n'}$ and a morphism $s : \Ical \to \Ical'$ such that 
\begin{itemize}
\item the diagram
\[\begin{tikzcd}[ampersand replacement=\&]
	{\Z^n} \& {\Z^{n'}} \\
	\Gamma \& {\Gamma'}
	\arrow["h"', from=1-1, to=2-1]
	\arrow["H", from=1-1, to=1-2]
	\arrow["{h'}", from=1-2, to=2-2]
	\arrow["L"', from=2-1, to=2-2]
\end{tikzcd}\]
commutes
    \item the diagram
   \[\begin{tikzcd}[ampersand replacement=\&]
	{\C^d} \& {\C^{d'}} \\
	{\Tscr_{h,\Ical}} \& {\Tscr_{h',\Ical'}}
	\arrow["{L_\C}", from=1-1, to=1-2]
	\arrow[from=1-1, to=2-1]
	\arrow[from=1-2, to=2-2]
	\arrow["{\lscr^{cal}}"', from=2-1, to=2-2]
\end{tikzcd}\]
commutes
\item For all $i \in \Ical$, $H(e_i)=e_{s(i)}$ 
\item For all $i \notin \Ical, H(e_i) \in \bigoplus_{j \notin \Ical'} \Z e_j$
\end{itemize}

\end{Def}



\subsection{Presented quantum tori}

The quantum tori as defined in the previous subsection are not suitable for the non-simplicial case since a non-simplicial cone can have more than $d$ generators. Hence, we have to consider quantum tori as a quotient stack of a higher-dimensional space $\C^p$. To do this, we will replace the group $\Gamma$ and its calibration $h$ by a subgroup $G$ of $\R^p$ isomorphic to it with a calibration $\varphi$ of it. Then, we describe these quantum tori by the data of the underlying standard quantum torus (to keep track of the combinatorial data of the group $\Gamma$) with the data of a stack isomorphism, respecting the virtual generators, with such quotient stack. More precisely, 
\begin{Def} \label{calibrated_qtorus_def}
A presented calibrated quantum torus is a 6-uple $(\Tscr^{cal}_{h,\Ical},\varphi: \Z^N \to G \subset \R^p,\Ical',L,H,s)$ where $\varphi$ is a calibration of the group $G$ (with $\Ical'$ its set of virtual generators), $L: \R^p\to \R^d$ is a linear epimorphism, $H: \Z^N \to \Z^N$ is a group isomorphism  and $s: \Ical \to \Ical'$ is a bijection such that:
\begin{itemize}
\item $L(G)=\Gamma$ and $L_{|G}: G \to \Gamma$ is a group isomorphism.
\item The diagram
\[\begin{tikzcd}[ampersand replacement=\&]
	{\Z^n} \& {\Z^n} \\
	G \& \Gamma
	\arrow[from=1-1, to=1-2]
	\arrow["\varphi"', from=1-1, to=2-1]
	\arrow["h", from=1-2, to=2-2]
	\arrow["{L_{|G}}"', from=2-1, to=2-2]
\end{tikzcd}\]
is commutative.
\item For all $i \in\Ical, H(e_i)=e_{s(i)}$ and for all $i \notin \Ical$, $H(e_i) \in \bigoplus_{j \notin \Ical'} \Z e_j$
\end{itemize} 
The morphism $\varphi$ is the calibration of this presented calibrated quantum torus.
\end{Def}

With these data, we can define the quotient stack  $[\C^p/\Z^N \times \ker(L \otimes_\R id_\C)]$, where $\Z^N \times \ker(L \otimes_\R id_\C)$ acts on $\C^p$ through
\[ (m,w) \cdot z=z+\varphi(m)+w,\]
More precisely, these data are equivalent to the datum of a stack isomorphism 
\[ [\C^p/\Z^N \times \ker(L \otimes_\R id_\C)]  \simeq \Tscr^{cal}_{h,\Ical}\]
which respects the combinatorial data. \\
If $L$ is a linear isomorphism then this stack isomorphism is a calibrated torus isomorphism as defined in definition \ref{torus_morphis_cal}.

We can define in the same way non-calibrated quantum tori: 
\begin{Def} \label{qtorus_def}
A presented non-calibrated quantum torus is a couple $(\Tscr_{d,\Gamma},L)$ where $\Tscr_{d,\Gamma}$ is the non-calibrated quantum torus associated to $\Gamma$ and $L$ is a linear epimorphism $\R^p \to \R^d$.

\end{Def}

In particular, the linear morphism descends to a stack isomorphism \[ [\C^p/L_\C^{-1}(\Gamma)] \simeq \Tscr_{d,\Gamma} \] 
(where $L_\C=L \otimes_\R id_\C$ and $L_\C^{-1}(\Gamma)$ acts on $\C^p$ by translations)  which is a quantum torus isomorphism if $L$ is a linear isomorphism as defined in \ref{torus_morphis_noncal}.

In both cases, we can define a multiplicative form of the torus thanks to the exponential map \eqref{exp_map} :
\begin{enumerate}
    \item In the non-calibrated case, we have an isomorphism (see diagram \eqref{exp_quotient_noncal}): 
    \[ [\C^p/L_\C^{-1}(\Gamma)] \simeq [\T^p/E(L_\C^{-1}(\Gamma))] 
    \]
    where $E(L_\C^{-1}(\Gamma))$ acts multiplicatively on $\T^p$.
    \item In the calibrated case, we have to suppose $h$ standard, i.e., we suppose that there is a subset $\widetilde{I}=\{i_1,\ldots,i_d\} \subset \{1,\ldots,p\}$ such that $\Z^{\widetilde{I}} \subset G$ and $h(e_k)=e_{i_k}$ for $k=1..d$. Then, the exponential map gives us:
    \[ [\C^p/\Z^N \times \ker(L_\C)] \simeq [\T^p/\Z^{N-d} \times E(\ker(L_\C))] \]
    where $\Z^{N-d} \times E(\ker(L_\C))$ acts on $\T^p$ through 
    \[ (m,E(w)) \cdot z=E(\varphi(0\oplus m)+w)z
    \]

\end{enumerate}

\begin{Def}
A morphism of presented non-calibrated quantum tori \[(\Tscr_{d,\Gamma},L) \to (\Tscr_{d',\Gamma'},L')\] or presented torus morphism is a couple $(\Lcal,\Lcal')$ of linear morphisms such that 
\begin{itemize}
    \item $\Lcal(\Gamma) \subset \Gamma'$
    \item The following diagram commutes
    \[\begin{tikzcd}[ampersand replacement=\&]
	{\C ^d} \& {\C^p/\ker(L_\C)} \\
	{\C^{d'}} \& {\C^{p'}/\ker(L'_\C)}
	\arrow["{\Lcal'}", from=1-2, to=2-2]
	\arrow["\Lcal"', from=1-1, to=2-1]
	\arrow["{[L_\C]}"', from=1-2, to=1-1]
	\arrow["{[L'_\C]}", from=2-2, to=2-1]
\end{tikzcd}\]
\end{itemize}

where $[L]$ is the morphism induced by $L$ on the quotient.
\end{Def}

\begin{Def} \label{morphism_presented_tori}
A morphism of presented calibrated quantum tori \[(\Tscr^{cal}_{h,\Ical},\varphi: \Z^N \to G \subset \R^p,\widetilde{\Ical},L,H,s) \to (\Tscr^{cal}_{h',\Ical'},\varphi': \Z^{N'} \to G' \subset \R^{p'},\widetilde{\Ical}',L',H',s')\]  or presented calibrated torus morphism  is a 6-uple $(\Lcal,\Hcal, \Scal, \Lcal ',\Hcal ', \Scal ')$ of morphisms such that 
\begin{itemize}
\item $(\Lcal,\Hcal,\Scal)$ induces a  torus morphism $\Tscr^{cal}_{h,\Ical} \to \Tscr^{cal}_{h',\Ical'}$,

\item  for all $j \in \widetilde{\Ical}, \Hcal'(e_j)=e_{\Scal'(j)}$ and for all  $j \notin \widetilde{\Ical}$, 
\[
\Hcal'(e_j) \in \bigoplus_{ k \notin \widetilde{\Ical}'} \Z e_k
\]

\item The following diagrams commute: 

\[\begin{tikzcd}[ampersand replacement=\&]
	{\C^n} \& {\C^n} \& {\C^d} \& {\C^p/\ker(L_\C)} \\
	{\C^{n'}} \& {\C^{n'}} \& {\C^{d'}} \& {\C^{p'}/\ker(L'_\C)}
	\arrow["{\Lcal'}", from=1-4, to=2-4]
	\arrow["{[L_\C]}"', from=1-4, to=1-3]
	\arrow["h", from=1-2, to=1-3]
	\arrow["H"', from=1-2, to=1-1]
	\arrow["{H'}", from=2-2, to=2-1]
	\arrow["{h'}"', from=2-2, to=2-3]
	\arrow["{[L'_\C]}", from=2-4, to=2-3]
	\arrow["\Lcal", from=1-3, to=2-3]
	\arrow["\Hcal", from=1-2, to=2-2]
	\arrow["{\Hcal'}"', from=1-1, to=2-1]
\end{tikzcd}\]

\[\begin{tikzcd}[ampersand replacement=\&]
	\Ical \& {\widetilde{\Ical}} \\
	{\Ical'} \& {\widetilde{\Ical}'}
	\arrow["s", from=1-1, to=1-2]
	\arrow["{\Scal'}", from=1-2, to=2-2]
	\arrow["{s'}"', from=2-1, to=2-2]
	\arrow["\Scal"', from=1-1, to=2-1]
\end{tikzcd}\]
\end{itemize}

\end{Def}

In particular, we have the following commutative diagram showing the relation between the different calibrations: 

\[\begin{tikzcd}[ampersand replacement=\&]
	{\Z^n} \& {\Z^n} \& {\Z^{n'}} \& {\Z^{n'}} \\
	G \& \Gamma \& {\Gamma'} \& {G'}
	\arrow["\varphi"', from=1-1, to=2-1]
	\arrow["{\varphi'}", from=1-4, to=2-4]
	\arrow["{H'^{-1}}", from=1-3, to=1-4]
	\arrow["\Hcal", from=1-2, to=1-3]
	\arrow["H", from=1-1, to=1-2]
	\arrow["h", from=1-2, to=2-2]
	\arrow["L"', from=2-1, to=2-2]
	\arrow["\Lcal"', from=2-2, to=2-3]
	\arrow["{h'}", from=1-3, to=2-3]
	\arrow["{(L'_{G'})^{-1}}"', from=2-3, to=2-4]
\end{tikzcd}\]
We can reformulate this in term of quotient stack:
\begin{Lemma} \label{equiv_6uple}
We use the same notations as definition \ref{morphism_presented_tori}. \\
Let $\lscr': [\C^p/\Z^{N} \times \ker(L_\C)] \to [\C^{p'}/\Z^{N'} \times \ker(L'_\C)]$ be the stack morphism described by the linear morphisms $(\Lcal',\Hcal')$ and $\lscr^{cal}: \Tscr^{cal}_{d,\Gamma} \to \Tscr^{cal}_{d',\Gamma'}$ be the stack morphisms described by the linear morphisms $(\Lcal,\Hcal)$. Then, the diagram

\[\begin{tikzcd}[ampersand replacement=\&]
	{[\C^p/\Z^n \times \ker(L_\C)]} \&\& {[\C^{p'}/\Z^{n'} \times \ker(L'_\C)]} \\
	{\Tscr_{h,\Ical}} \&\& {\Tscr_{h',\Ical'}}
	\arrow["{\lscr'}", from=1-1, to=1-3]
	\arrow["\simeq"', from=1-1, to=2-1]
	\arrow["\simeq", from=1-3, to=2-3]
	\arrow["{\lscr^{cal}}"', from=2-1, to=2-3]
\end{tikzcd}\]
commutes (where the isomorphisms are the stack isomorphisms encoded by the presentation of the presented calibrated quantum tori). 
\end{Lemma}

\begin{Rem}
By definition, the category of presented quantum tori is a subcategory of the arrow category of quantum tori which is the category whose objets are torus morphisms and the morphisms between them are the diagrams of the form 
\[\begin{tikzcd}
	{\Tscr_{0,0}} & {\Tscr_{1,0}} \\
	{\Tscr_{0,1}} & {\Tscr_{1,1}}
	\arrow[from=1-1, to=1-2]
	\arrow["{f_0}"', from=1-1, to=2-1]
	\arrow["{f_1}", from=1-2, to=2-2]
	\arrow[from=2-1, to=2-2]
\end{tikzcd}.\]

\end{Rem}

The presented quantum tori are (essentially) not new quantum tori. In fact, all presented quantum tori with the same underlying quantum tori are isomorphic:

\begin{Lemma} \label{isom_torus}
$(\Tscr^{cal}_{h,\Ical},\varphi: \Z^N \to G \subset \R^p,\Ical',L,H,s)$ and $(\Tscr^{cal}_{h,\Ical},h: \Z^N \to \Gamma,\Ical,id_{\R^d},id_{\R^N},id_\Ical)$ (resp.  $(\Tscr_{d,\Gamma},L)$ and $(\Tscr_{d,\Gamma},id)$) are isomorphic.

\end{Lemma}

\begin{proof}
An isomorphism is given by $(id_{\C^d},H,s,[L_\C],id_{\C^N},id_{\Ical'})$ (resp. $(id_{\C^d},[L_\C])$).

\end{proof}

\begin{Prop} \label{forgetful_isom}
The forgetful functor \[(\Tscr^{cal}_{h,\Ical},\varphi: \Z^N \to G,\Ical',L,H,s) \to \Tscr^{cal}_{h,\Ical},\] resp. $(\Tscr_{d,\Gamma},L) \to \Tscr_{d,\Gamma}$, is an equivalence of categories between the category of presented calibrated quantum tori and the category of standard calibrated quantum tori, resp. between the category of non-calibrated quantum tori and the category of non-calibrated standard quantum tori.

\end{Prop}

\begin{proof}
Use lemma \ref{isom_torus}
\end{proof}

\renewcommand{\theThm}{\arabic{section}.\arabic{subsection}.\arabic{subsubsection}.\arabic{Thm}}
\section{Definition of (non-simplicial) quantum toric varieties}
\label{NSQTV}
Let $\Gamma$ be a finitely generated subgroup of $\R^d$ such that $\Vect_\R(\Gamma)=\R^d$, $(\Delta,h^{cal} \colon \Z^N \to \Gamma, \Ical)$ a standard calibrated quantum fan (with $\Ical$ its set of virtual generators). In order to define a quantum variety associated to this fan, we will describe the affine quantum variety associated to each cone of this fan. The crucial point is to replace the calibration $h$ of $\Gamma$ by an adapted calibration $\varphi$ for each maximal cone $\sigma$ of the fan $\Delta$.

After that, we will define the quantum toric varieties with the descent data of affine pieces as \citep{VQS}. 

At the end of this section, we will prove that the correspondence \[(\Delta,h^{cal} \colon \Z^N \to \Gamma, \Ical) \mapsto \Xscr^{cal}_{\Delta,h^{cal},\Ical}\] is an equivalence of categories.
\subsection{ Affine quantum toric varieties}
Let
\[\sigma=\sigma_I=\Cone(v_{i_1}\coloneqq h^{cal}(e_{i_1}),\ldots,v_{i_p}\coloneqq h^{cal}(e_{i_p}))\] (with $I\coloneqq \{i_1,\ldots,i_p\} \subset \{1,\ldots,N\})$ be a strongly convex cone of $\Delta$. There are two possibilities: $\sigma$ is of maximal dimension i.e. of dimension $d$ or not. In the latter, we have to do a choice of a completion of $\Vect_\R(\sigma)$ in $\R^d$. Fortunately, all obtained quantum varieties are isomorphic and this isomorphism respects a cocycle condition. 
\subsubsection{Cones of dimension $d$}
\label{calibrated_case}
Suppose $\sigma$ is a cone of dimension $d$ i.e. $\Vect_\R(\sigma)=\R^d$. \\
Note $h_{\sigma}: \Z^I \to \Gamma$ the restriction of $h^{cal}$ on $\Z^I$, $h_{\sigma\C}: \C^I \to \C^d$ the $\C$-linear map associated to it and $\widehat{\sigma}$ the cone of $\R^N$ defined by  
\begin{equation} \label{cone_Rr}
\widehat{\sigma}=\Cone(e_i, i \in I)
\end{equation}
and $\widetilde{\sigma}$ its restriction on $\R^I$

Let $\Bcal=(v_i, i\in \widetilde{I})$ a subfamily of $(v_1,\ldots,v_p)$ which is a basis of $\C^d$. 
Then, we will consider the decomposition $\C^I=\C^{\widetilde{I}} \oplus \ker(h_{\sigma\C})$. The map $h_{\sigma\C}$ induces a linear isomorphism $\psi$ between $\C^{\widetilde{I}}$ and $\C^d$. \\
Let $\chi \in \Sfrak_N$ be a permutation such that $\chi(\{1,\ldots,d\})=\widetilde{I}$. 

\begin{Def}
The linear morphism $\varphi: \C^N \to \C^{\widetilde{I}} \hookrightarrow \C^I$ defined by 
\[e_k \mapsto \psi^{-1}(h^{cal}_\C(e_{\chi(k)}))\]
is called calibration associated to $\sigma$ (and $\Bcal$ and $\chi$)
\end{Def}

We can think this morphism as a calibration induced by the calibration $h$ on $\C^{I}$. Indeed, the image of $\Z^d \oplus 0$ by the morphism $\varphi$ is $\Z^{\widetilde{I}}$ (by construction) and the group $\varphi(\Z^N)$ is isomorphic to $\Gamma$.

We can define an action of $\Z^{N-d} \times E(\ker(h_{\sigma\C}))$ on $\C^I$ by setting for $(m,E(t)) \in \Z^{N-d} \times E(\ker(h_{\sigma\C}))$ and $z \in \C^I$,
\[ (m,t) \cdot z=E(\varphi(0 \oplus m))+t)z\]
\begin{Rem}
  The non-effectiveness of this action only comes from the calibration i.e. the subgroup of ineffectivity of this action is $\ker(h^{cal})$.   
\end{Rem}

The action of $\T^I$ on $\C^I=U_{\widetilde{\sigma}}$ commutes with this action. Hence, we can form the quotient $[\C^I/\Z^{N-d} \times E(\ker(h_{\sigma\C}))]$. 

\begin{Lemma}
The quotient stack $[\C^I/\Z^{N-d} \times E(\ker(h_{\sigma\C}))]$ depends neither on the choice of the basis $\Bcal$ nor on the permutation $\chi$. Hence, the action induced by the morphism $\varphi$ depends only on $\sigma$.
\end{Lemma}
\begin{proof}

Let $(\Bcal=(v_i, i\in \widetilde{I}),\chi)$ and $(\Bcal'=(v_i, i\in \widetilde{I}'),\chi')$ be two pairs of basis and permutation. Then we can define two associated isomorphisms $\psi: \C^{\widetilde{I}} \to \C^d$ and $\psi': \C^{\widetilde{I}'} \to \C^d$ and two morphisms $\varphi: \C^N \to \C^I$, $\varphi': \C^N \to \C^{I'}$ . Then, the identity is an $\alpha$-equivariant morphism, where $\alpha$ is the morphism $\Z^{N-d} \times E(\ker(h_{\sigma\C})) \to \Z^{N-d} \times E(\ker(h_{\sigma\C}))$ defined by 
\[
\alpha(m,E(y))=(P_{\chi'}^{-1}P_\chi(m),E(\varphi(m)-\psi'^{-1}h(\varphi(m)))E(y))
\]
where $P_\chi$ is the linear map associated to $\chi$.

Indeed, the following diagram is commutative 

\[\begin{tikzcd}[ampersand replacement=\&]
	{E(h_\sigma^{-1}(\Gamma))} \&\& {E(h_\sigma^{-1}(\Gamma))} \\
	{\Z^{n-d}\times E(\ker(h_{\sigma\C}))} \&\& {\Z^{n-d}\times E(\ker(h_{\sigma\C}))}
	\arrow["{=}", from=1-1, to=1-3]
	\arrow["{\text{action through } \varphi}"', from=1-1, to=2-1]
	\arrow["{\text{action through } \varphi'}", from=1-3, to=2-3]
	\arrow["\alpha"', from=2-1, to=2-3]
\end{tikzcd}\]

\end{proof}

\begin{Def}
The stack $\Uscr_\sigma^{cal}\coloneqq [\C^I/\Z^{N-d} \times E(\ker(h_{\sigma\C}))]$ is the quantum toric variety associated to the cone $\sigma$ and to the calibration $h^{cal}: \Z^N\to \Gamma$.
\end{Def}

\begin{Rem}[Local models]
If $\sigma$ is non-simplicial (i.e. $\ker(h_{\sigma\C})$ is non-empty and thus is not discrete) then the quantum toric variety associated to $\sigma$ cannot be described as a quasifold (i.e. locally the quotient of a space $\R^n$ by a discrete subgroup, see \citep{quasifold}) in contrast to the simplicial case. \\
It is the product of such (stacky version of) quasifolds with the quotient of $\Theta_{|I|-d} \coloneqq [\C^{|I|-d}/(\C^*)^{|I|-d}]$. This last stack can be described as follows : 
an open point (given by the orbit $(\C^*)^{|I|-d}$), locally closed points (given by the orbits of the form $(\C^*)^J \times 0$ where $0<|J|<|I|-d$) and a closed point (given by the orbit $\{0\}$). Moreover, the closure of the point $x_J \coloneqq (\C^*)^J \times 0$ is 
\[
\overline{x_J}=[\C^J \times 0/(\C^*)^{|I|-d}]=\bigcup_{K \subset J} [(\C^*)^K/(\C^*)^{|I|-d}]=\bigcup_{K \subset J} x_K
\]
Hence, we find the stratification by quasifolds proved in \cite{battaglia2008geometric}.
\end{Rem}

The associated "torus" to $\Uscr_{\sigma}^{cal}$ is the stack $[\T^I/\Z^{N-d} \times E(\ker(h_{\sigma\C}))]$. We can endow this stack with a structure of presented calibrated quantum torus thanks to the morphisms used to define $\Uscr_{\sigma}^{cal}$:

\begin{Prop}
$(\Tscr_{h,\Ical}^{cal}, \varphi: \Z^N \to \psi^{-1}(\Gamma),h_{\sigma\C},\chi^{-1}(\Ical),P_\chi: e_i \mapsto e_{\chi(i)},\chi )$ is a presented calibrated quantum torus which encodes the stack $[\T^I/\Z^{N-d} \times E(\ker(h_{\sigma\C}))]$
\end{Prop}

\begin{proof}
This proposition comes from the fact that $\psi$ is induced by $h_{\sigma\C}$ on $\C^I/\ker(h_{\sigma\C}) \simeq \C^{\widetilde{I}}$ and the equality $\varphi=\psi^{-1} h P_\chi$.
\end{proof}

\begin{Not}
In what follows, we will omit the isomorphism and just write $[\T^p/\Z^{N-d} \times E(\ker(h_{\sigma\C}))]$ instead of this 6-uple. 
\end{Not}

\begin{Def}
An affine quantum toric variety is a stack $\Uscr$ with a dense open substack $\Tscr$ which is isomorphic (as stack) to an affine quantum toric variety associated to a cone and a calibration $(h,\Ical)$ and such that this isomorphism restricts to a isomorphism between $\Tscr$ and $\Tscr^{cal}_{h,\Ical}$.
\end{Def}
\begin{War}
If $\Gamma$ is a lattice, the stack $\Uscr_\sigma^{cal}=[\C^I/\ker(h_{\sigma\C})] \times \mathscr{B}\Z^{N-d}$ is not a (stack representable by a) variety (due to the $\mathscr{B} \Z^{N-d}$-part). Moreover, if we forget this $\mathscr{B}\Z^{N-d}$-part, we get a non-representable stack too.
\end{War}

However, if we replace the stack quotient by the GIT quotient, we get the classical toric variety associated to $\sigma$: 

\begin{Prop} \label{isom_non-simplicial}
Let $h : \Z^n \to \Gamma$ be a calibration (with $\Ical$ its set of virtual generators). \\
If $\Gamma$ is discrete and the set $\Ical$ is empty,
then we can define the (classical) toric variety $U_\sigma$ associated to $\sigma$ (and $\Gamma$) and
\[U_\sigma=\C^I\sslash E(\ker(h_{\sigma\C})) \]

where $\sslash$ denotes the categorical quotient\footnote{A categorical quotient of a variety $X$ with action of a Lie group $G$ is a $G$-invariant morphism $\pi \colon X \to Y$ such that every $G$-invariant morphism $X \to Z$ (uniquely) factors through $\pi$.}.
\end{Prop}

\begin{proof}

In \citep{cox} (theorem 5.1.11), the authors prove that $\C^I \to U_\sigma$ is an almost geometric quotient for the action of $G$ defined by \[G\coloneqq \Hom_\Z(\coker(M=(v_i, i \in I)^T: \Z^I \to \Z^d),\C^*)\]
(The cokernel in this expression is also defined as the class group of the variety $U_\sigma$ see \citep{cox} p172/173).
We will prove that the group $G$ is isomorphic to $E(\ker(h_{\sigma \C}))$. To do this, we will prove that they are both isomorphic to the group $h_{\sigma\C}^{-1}(\Gamma)$: \\
This group $h_{\sigma\C}^{-1}(\Gamma)$ can be written as a direct sum: 
\[
h_{\sigma\C}^{-1}(\Gamma)=\Gamma_0 \oplus \ker(h_{\sigma\C})
\]
where $\Gamma_0$ is a subgroup of $\C^I$ isomorphic to $\Gamma$. \\
If we suppose $\Gamma$ discrete, we can suppose too that $\Gamma = \Z^d$ without loss of generality. Then, the subgroup $\Gamma_0$ is a subgroup of $\Z^I$ and we have the equality 
\[
E(h_{\sigma\C}^{-1}(\Gamma))=E(\ker(h_{\sigma\C})).
\]

We will now prove that we have a group isomorphism $G \simeq E(h_{\sigma\C}^{-1}(\Gamma))$:\\
Since $\C^*$ is a divisible group (i.e. an injective abelian group) , we have a exact sequence:
\begin{equation}\label{sec_discret}
   \begin{tikzcd}[ampersand replacement=\&]
	0 \& G \& {\T^I} \& {\T^d} \& 0
	\arrow[from=1-1, to=1-2]
	\arrow[from=1-2, to=1-3]
	\arrow["\varphi", from=1-3, to=1-4]
	\arrow[from=1-4, to=1-5]
\end{tikzcd}
\end{equation} 
where the morphism $\varphi$ is the induced morphism by $\Hom(M,\C^*) $ and the isomorphisms $\Hom(\Z^I,\C^*) \to \T^I$,  $(u: \Z^I \to \C^*)\mapsto (u(e_i),i \in I)$ and $\Hom(\Z^d,\C^*) \to \T^d$,  $(u: \Z^d \to \C^*)\mapsto (u(e_1),\ldots,u(e_d))$.

Then, we deduce from the equality $h_{\sigma\C}=M_\C^T$ that $\varphi$ is the morphism $\T^I \to \T^d$ induced by $h$. Finally, thanks to the exact sequence \eqref{sec_discret}, we get the isomorphism $G \simeq E(h_{\sigma\C}^{-1}(\Z^d))
$ (since $\ker(E)=\Z^d)$ \\
The categorical quotient condition is proved in the theorem 5.1.9 and the part a) of theorem 5.1.11 of \citep{cox}.
\end{proof}
We have the equality \[\C^I\sslash E(\ker(h_{\sigma\C})) \times \mathscr{B}\Z^{n-d}=[\C^I/\Z^{N-d} \times E(\ker(h_{\sigma\C}))]\] 
if, and only if, $\sigma$ is a simplicial cone. Indeed, the categorical quotient (for the action of $G$) $\C^I\to U_\sigma$ is geometric if, and only if, $\sigma$ is a simplicial (see \citep{cox} theorem 5.1.11 b)).  
\begin{Ex} \label{ex_max}
Let \[\Gamma=\Z e_1+\Z e_2+\Z e_3+\Z(ae_1-be_2+ce_3) +\sum_{i=5}^N \Z v_i \subset \R^3\] with $a,b,c \in \R_{>0},v_i \in \R^3$ be a subgroup of $\R^3$, $h^{cal}: \Z^N \to \Gamma$ be the calibration of $\Gamma$ defined by $h^{cal}(e_i)=e_i$ for $i=1,2,3$, $h^{cal}(e_4)=ae_1-be_2+ce_3\eqqcolon v$ and $h^{cal}(e_i)=v_i$ for $i \geq 5$. Let $\sigma=\mathrm{Cone}(e_1,e_2,e_3,ae_1-be_2+ce_3)$ be a strongly convex cone of $\R^3$.

\begin{figure}[h]
\centering
\caption{The cone $\sigma$}
\begin{tikzpicture}[scale=2]
    \draw (0,0) to (-1,1) node[left]{$e_2$} to (0,2) node[above]{$e_3$} to (1,1.5) node[above right]{$v$} to (1,0.5) to (0,0);
    \draw (0,0) to (0,2);
    \draw[dashed] (-1,1) to (1,0.5) node[right]{$e_1$};
    \draw (0,0) to (1,1.5);
\end{tikzpicture}
\end{figure}

The morphism $h_\sigma$ is the restricted map $h^{cal}_{|\Z^4 \oplus 0}$ and the kernel $\ker(h_{\sigma\C})$ is the line
\[\ker(h_{\sigma\C})=\C(-a,b,-c,1)\]

We have a decomposition $\C^4=(\C^3 \oplus 0) \oplus \ker(h_{\sigma\C})$ and hence, an action of $\Z^{N-3}\times E(\ker(h_{\sigma \C}))$ on $\C^4$ defined by: 
\begin{equation} \label{action_ex}
    (m,E(t)) \cdot s=E(h(0 \oplus m)+t)s
\end{equation}
Then, $\Uscr_{\sigma}^{cal}=[\C^4/\Z^{N-3} \times E(\C(-a,b,-c,1))]$ 

The action of $\Z^{N-3} \times E(\C(-a,b,-c,1))$ on $\C^4$ is not effective. Indeed, $\ker(h^{cal})$ stabilizes every point of $\C^4$. Moreover, there are some points with bigger stabilizer:
\begin{itemize}
    \item The stabilizer of $0_{\C^4}$ is $\Z^{N-3} \times E(\C(-a,b,-c,1))$ ;
    \item The stabilizer of $(\C^3 \setminus \{0\} \oplus 0) \oplus 0$ is $\ker(h^{cal}) \times E(\ker(h_{\sigma\C})$ ;
    \item The stabilizer of $0 \oplus \C^*(-a,b,-c,1)$ is $\Z^{n-d} \times \{0\}$ ;
    \item The stabilizer of the other points are $\ker(h^{cal})$.
\end{itemize}

For the case $N=4$ and $a=b=c=1$, we find a stacky version of the toric variety $\C^4\sslash G=V(yt-xz) \subset \C^4$ (see \citep{fulton} p17)
\end{Ex}
\subsubsection{Cone of dimension $k<d$}
\label{cone_non_max}
Suppose $\sigma=\sigma_I \subset \R^d$ is a cone of dimension $k<d$. \\Let $J$ a subset of $[\![1,N]\!]$ of cardinal $d-k$ such that 
\[
\R^d=\Vect(\sigma) \oplus \Vect(v_j, j \in J)
\]
Hence, $\overline{h}_\sigma: \C^I \oplus \C^J \to \C^d, e_i \mapsto v_i$ is an epimorphism. 
Now, we can reuse the discussion of subsection \ref{calibrated_case} and define an action of $\Z^{N-d} \times E(\ker(\overline{h}_{\sigma\C}))$ on $\C^I \oplus \C^J$ (we will suppose that the permutation $\chi$ sends $\{k+1,\ldots, d\}$ on $J$). We remark that the toric variety $\C^I \times \T^J=U_{\widetilde{\sigma} \times 0}$ is preserved by this action and thus, we can define the affine quantum toric variety associated to $\sigma$: 
\[\Uscr^{cal}_\sigma\coloneqq [\C^I \times \T^J/\Z^{N-d} \times \ker(\overline{h}_{\sigma})] \]

\begin{Prop}
The affine quantum toric variety $\Uscr^{cal}_\sigma$ is well-defined i.e. if we change the family $(v_j, j \in J)$, we get an isomorphism which respects a cocycle condition.

\end{Prop}
\begin{proof}
Let $J$ and $J'$ be two subsets of $[\![1,N]\!]$ of cardinal $d-k$ such that $\dim(\Vect(v_j, j \in J))=\dim(\Vect(v_j, j' \in J'))=d-k$ and $\chi$, $\chi'$ be the associated permutations (and $P_\chi$, $P_{\chi'}$ the associated linear maps). The morphism $P_{\chi'} \circ P_\chi^{-1}: \C^J \to \C^{J'}$ is an isomorphism and $id \oplus P_{\chi'} \circ P_\chi^{-1}: \C^I \oplus \C^J \to \C^I \oplus \C^{J'}$ too. 

We will prove that this morphism induces an isomorphism of stacks
\[[\C^I \times \T^J/\Z^{N-d} \times \ker(h_{\sigma\C} \oplus h_J)] \simeq [\C^I \times \T^{J'}/\Z^{N-d} \times \ker(h_{\sigma\C} \oplus h_{J'})] \]  
where $h_J$ (resp. $h_{J'}$) is the restriction of $h^{cal}_\C$ on $\Z^J$ (resp. $\Z^{J'}$)

Firstly, we can remark that $x \oplus y \in \ker(h_{\sigma\C} \oplus h_J)=\ker(h_{\sigma\C} \oplus h_{J'})$ if, and only if, $x \in \ker(h_{\sigma\C})$ and $y=0$. Thus, we get the following commutative diagram (every arrow is an isomorphism):

\[\begin{tikzcd}[ampersand replacement=\&]
	{\C^I \oplus \C^J/\ker(h_{\sigma\C} \oplus h_J)} \&\& {\C^I \oplus \C^{J'}/\ker(h_{\sigma\C} \oplus h_{J'})} \\
	{\C^I/\ker(h_{\sigma}) \oplus \C^J} \&\& {\C^I/\ker(h_{\sigma}) \oplus \C^{J'}}
	\arrow[from=1-1, to=2-1]
	\arrow["{id \oplus (P_{\chi'} \circ P_\chi^{-1})}", from=1-1, to=1-3]
	\arrow[from=2-1, to=2-3]
	\arrow[from=1-3, to=2-3]
\end{tikzcd}\]

By construction, the isomorphism $id \oplus P_{\chi'}^{-1} \circ P_\chi$ descends to quotient. Hence, we get:  
\[\begin{tikzcd}[ampersand replacement=\&]
	{\C^I \oplus \C^J} \& {\T^I \times\T^J} \& {\C^I \times\T^J} \& {[\C^I \times \T^J/\Z^{n-d}\times\ker(h_{\sigma\C}\oplus h_J)]} \\
	{\C^I \times \C^{J'}} \& {\T^I \times \T^{J'}} \& {\C^I \times \T^{J'}} \& {[\C^I \times \T^{J'}/\Z^{n-d}\times\ker(h_{\sigma\C}\oplus h_J)]}
	\arrow["{id \oplus P_{\chi'}P_\chi^{-1}}"', from=1-1, to=2-1]
	\arrow["E", from=1-1, to=1-2]
	\arrow["E"', from=2-1, to=2-2]
	\arrow[hook, from=1-2, to=1-3]
	\arrow[hook, from=2-2, to=2-3]
	\arrow[from=1-3, to=1-4]
	\arrow[from=2-3, to=2-4]
	\arrow["{\hscr_{J'J}}", from=1-4, to=2-4]
	\arrow[from=1-3, to=2-3]
	\arrow[from=1-2, to=2-2]
\end{tikzcd}\]
The rightmost morphism is the desired morphism, which is an isomorphism. 

If we take a third subset $J''$ (and a permutation $\chi''$) , we get two other isomorphisms $\hscr_{J''J}$ and $\hscr_{J''J'}$. The equality \[(id \oplus P_{\chi''} \circ P_{\chi'}^{-1}) \circ (id \oplus P_{\chi'} \circ P_\chi^{-1})=id \oplus P_{\chi''} \circ P_\chi^{-1} \] induces an equality between the stack (iso)morphisms
\[\hscr_{J''J'} \circ \hscr_{J'J}=\hscr_{J''J} \]
\end{proof}

The following proposition is proved by the same manner as proposition \ref{isom_non-simplicial}.
\begin{Prop}
If $\Gamma$ is discrete then \[U_\sigma=\C^I \times \T^J\sslash E\left(\ker\left(\overline{h}_{\sigma}\right)\right) \]
\end{Prop}

\begin{Ex}
We will describe a variant of example \ref{ex_max}: 

Let 
\[ \Gamma=\Z e_1+\Z e_2+\Z e_3+\Z (ae_1-be_2+ce_3)+\sum_{i=5}^N v_i \Z \subset \R^4
\]
be a subgroup of $\R^4$ such that $\Vect_\R(\Gamma)=\R^4$, $h \colon \Z^N \to \Gamma $ be the calibration of $\Gamma$ defined by $h(e_i)=e_i$ for $i=1,2,3$, $h(e_4)=ae_1-be_2+ce_3$ and $h(e_i)=v_i$ for $i \geq 5$, 
and $\sigma=\Cone(e_1,e_2,e_3,ae_1-be_2+ce_3)$. 

Note $k \in \{5,\ldots,N\}$ an integer such that $(e_1,e_2,e_3,v_k)$ is a basis of $\C^4$. The kernel of the morphism $\overline{h_\sigma} \colon \C^4 \oplus \C^{\{k\}} \to \C^4$ is $\C(-a,b,-c,1,0)$ i.e. $\ker(h_{\sigma\C}) \oplus 0$. Then, $\Uscr_\sigma^{cal}=[\C^4 \times \C^*/\Z^{N-3} \times E(\C(-a,b,-c,1,0))]$
\end{Ex}

We will conclude this subsection with the compatibility of the construction with the restriction to a face of a cone.
\begin{Prop} \label{souscone}
Let $\sigma=\sigma_I$ be a cone and let $\tau=\sigma_{I'}$ be a face of $\sigma$. Then we have an isomorphism 
\[
\Uscr^{cal}_\tau \simeq [\C^{I'} \times \T^{I \setminus I'} \times \T^J/\Z^{N-d} \times E(\ker(\overline{h}_{\sigma}))] \hookrightarrow \Uscr^{cal}_\sigma
\]
which restricts to a torus isomorphism 
\[[\T^{I'}\times \T^{J'}/\Z^{N-d} \times E(\ker(\overline{h}_{\tau}))] \simeq [\T^{I'} \times \T^{I \setminus I'} \times \T^J/\Z^{N-d} \times E(\ker(\overline{h}_{\sigma}))] \]
induced by the identity of $\Tscr^{cal}_{h,\Ical}$.
\end{Prop}

\begin{proof}
We have the following commutative diagram 

\[\begin{tikzcd}[ampersand replacement=\&]
 	{[\T^{I'}\times\T^{J'}/\Z^{n-d} \times E(\ker(\overline{h}_\tau))]} \&\& {[\T^{I'}\times\T^{J'}/\Z^{n-d} \times E(\ker(\overline{h}_\tau))]} \\
	\& {\Tscr_{h,\Ical}}
	\arrow["{(\overline{h}_\tau,P_\chi)}"', from=1-1, to=2-2]
	\arrow["\lscr^{cal}", from=1-1, to=1-3]
	\arrow["{(\overline{h}_\sigma,P_{\chi'})^{-1}}"', from=2-2, to=1-3]
\end{tikzcd}\]
Since $h_\sigma(e_i)=h_\tau(e_i)$ for $i \in I'$ (hence $[h_\sigma]^{-1} \circ [h_\tau]([e_i])=[e_i]$)  then $\lscr^{cal}$ can be extended to an isomorphism $\Uscr_\tau^{cal} \to [\C^{I'} \times \T^{I \setminus I'} \times \T^J/\Z^{N-d} \times E(\ker(\overline{h}_{\sigma}))]$
\end{proof}
\subsubsection{Toric morphisms of affine quantum toric varieties}
\label{affine_morphisms}
We will use the same definition of toric morphism as \citep{VQS}:
\begin{Def} \label{morphism_cal}
A toric morphism between the two affine quantum toric varieties $\Uscr^{cal}_{\sigma_1}=[\C^I \times \T^J/G_1]$ and $\Uscr^{cal}_{\sigma_2}=[\C^{I'} \times \T^{J'}/G_2]$ is a stack morphism $[\C^I \times \T^J/G_1] \to [\C^{I'} \times \T^{J'}/G_2]$ which restricts to a presented calibrated torus morphism $[\T^{I} \times \T^J/G_1] \to [\T^{I'} \times \T^{J'}/G_2]$
\end{Def}
By the proposition \ref{forgetful_isom} and the definition of torus morphism, such torus morphism induces a torus morphism between standard quantum tori $\Tscr^{cal}_{h,\Ical} \to \Tscr^{cal}_{h',\Ical'}$ and thus two linear maps $(L: \R^d \to \R^{d'},H: \R^N \to \R^{N'})$. Moreover, $L(\sigma) \subset \sigma'$ and $H(\widehat{\sigma}) \subset \widehat{\sigma}'$. These morphisms forms a calibrated quantum fans morphism. Conversely, we will discuss how we can associate to a morphism of calibrated quantum fans a toric morphism: 

Let $\sigma$ be a cone of dimension $k$ of $\R^d$, $\sigma'$ be a cone of dimension $k'$ of $\R^{d'}$,   $\widetilde{\sigma}$  and $\widetilde{\sigma}'$ be the associated cone in $\R^N$ and $\R^{N'}$. \\
Let $(L,H): (\sigma,h) \to (\sigma',h')$ be a calibrated quantum fan morphism. \\
By definition of calibrated quantum fans morphism, there exists a cone $\widetilde{\sigma}'=\widetilde{\sigma_{I'}}$ of $\Delta'_{h'}$ such that $H(\widetilde{\sigma}) \subset \widetilde{\sigma}'$. \\
We will adapt the construction of \citep{VQS} (section 5.1). Firstly, we begin by replacing the calibration $h$ by the morphism $\varphi$ (used in definition in $\Uscr_{\sigma}^{cal}$) in diagram \eqref{qf_diag}: \\
Let $J$ be a subset of cardinal $d-k$ of $\{1,\ldots,N\}$ such that \[\C^d=\Vect_\C(\sigma) \oplus \Vect_\C(v_j,j \in J).\] Let $\widetilde{J}$ be a subset of $J$ such that for all $j \in \widetilde{J}$, $L(v_j) \notin L(\Vect_\C(\sigma))$ and such that the family $(L(v_j), j \in \widetilde{J})$ is free. Let $J'$ be a subset of cardinal $d'-\dim(\sigma')$ of $\{1,\ldots,N\}$ containing $\widetilde{J}$  such that \[\C^{d'}=\Vect_\C(\sigma') \oplus \Vect_\C(v'_j,j \in J').\] Note $h_J: \C^J \to \C^d$ (resp. $h_{J'}: \C^{J'} \to \C^{d'}$)  the linear map $e_j \mapsto v_j$ for all $j \in J$ (resp. $e_j \mapsto v'_j$ for all $j \in J'$).   
To sum up, we have the following commutative diagram: 
\begin{equation} \label{morphism_cone}
\begin{tikzcd}[ampersand replacement=\&]
	{\C^I \oplus \C^J} \&\& {\C^d} \\
	{\C^{I'} \oplus \C^{J'}} \&\& {\C^{d'}}
	\arrow["{h_{\sigma\C}+h_J}", from=1-1, to=1-3]
	\arrow["L", from=1-3, to=2-3]
	\arrow["{h_{\sigma'\C}+h_{J'}}"', from=2-1, to=2-3]
	\arrow["{\widetilde{L}}"', from=1-1, to=2-1]
\end{tikzcd}
\end{equation}
where $\widetilde{L} $ is the map $\C^I \oplus \C^J =\ker(h_{\sigma\C}) \oplus (\C^{\widetilde{I}} \oplus \C^J) \to \ker(h'_{\sigma'\C}) \oplus (\C^{\widetilde{I}'} \oplus \C^{J'})$ defined by :
\[
\forall w \in \ker(h_{\sigma\C}), \forall z \in \C^{\widetilde{I}}\oplus \C^J, \widetilde{L}(w,z)=(H(w),\psi'^{-1}(L(h_{\sigma\C}(z))))
\]
where $\psi'$ is the map induced by $h'_{\sigma'}$ used in the definition of $\Uscr_\sigma^{cal}$. This map is well defined since $H(\ker(h_{\sigma\C})) \subset \ker(h'_{\sigma'\C})$ since $H(\widehat{\sigma}) \subset \widehat{\sigma}'$ and $(L,H)$ is a morphism of calibrated quantum fans.

Let $\chi$ and $\chi'$ be permutations used to define the toric varieties $\Uscr^{cal}_\sigma$ and $\Uscr^{cal}_{\sigma'}$ i.e. permutations such that 
\[\chi(\{1,\ldots,k\})=\widetilde{I},\chi(\{k+1,\ldots,d\})=J\]
and 
\[\chi'(\{1,\ldots,k'\})=\widetilde{I'},   \chi'(\{k'+1,\ldots,d'\})=J'.\]
Let $P_\chi \in \GL_N(\R)$ and $P_{\chi'} \in \GL_{N'}(\R)$ be the associated linear maps.

Moreover, we can extend this diagram with the morphisms $\varphi=\psi^{-1} \circ  h_\C \circ P_\chi \colon \C^N \to \C^{I} \oplus \C^J$ and $\varphi'=\psi'^{-1} \circ  h'_\C \circ P_{\chi'} \colon \C^{N'} \to \C^{I'} \oplus \C^{J'}$ used to define the quantum toric varieties $\Uscr_\sigma^{cal}$ and $\Uscr_{\sigma'}^{cal}$ (and more precisely, used to define an action of $\Z^{N-d}$ (resp. $\Z^{n'-d'}$) on $\C^I \oplus \C^J$  (resp. $\C^{I'} \oplus \C^{J'}$): 
\begin{equation} \label{morphism_cone_extended}
\begin{tikzcd}[ampersand replacement=\&]
	{\C^N} \&\& {\C^I \oplus \C^J} \&\& {\C^d} \\
	{\C^{N'}} \&\& {\C^{I'} \oplus \C^{J'}} \&\& {\C^{d'}}
	\arrow["{h_{\sigma\C}+h_J}", from=1-3, to=1-5]
	\arrow["L", from=1-5, to=2-5]
	\arrow["{h_{\sigma'\C}+h_{J'}}"', from=2-3, to=2-5]
	\arrow["{\widetilde{L}}"', from=1-3, to=2-3]
	\arrow["\varphi", from=1-1, to=1-3]
	\arrow["{\varphi'}"', from=2-1, to=2-3]
	\arrow["{P_{\chi'}^{-1}HP_\chi}"', from=1-1, to=2-1]
\end{tikzcd}
\end{equation}
After this replacement, we can follow the construction of corollary 5.6 of \citep{VQS} (i.e. the section 5.1) in order to associate to the morphisms $(H_{\chi\chi'}\coloneqq P_{\chi'}^{-1} H P_\chi,\widetilde{L})$ a toric morphism $\Uscr_{\sigma}^{cal} \to \Uscr_{\sigma'}^{cal}$: 

Note $E_I: \C^I \times \C^J \to \T^I \times \C^J$ the map defined by:
\[((z_i)_{i \in I},(w_j)_{j \in J}) \mapsto ((E(z_i))_{i \in I},(w_j)_{j \in J}) \]
and $\overline{E}_{J}: \C^I \times \C^J \to \C^I \times \T^J$ the map defined by: 
\[((z_i)_{i \in I},(w_j)_{j \in J}) \mapsto ((z_i)_{i \in I},(E(w_j))_{j \in J}) \]
In the same way, we define $E'_{I'},\overline{E}'_{J'}$ and $E'=E'_{I'}\circ \overline{E}'_{J'}$)
Then, we have the commutative diagram: 
\begin{equation} \label{L_extent}
\begin{tikzcd}[ampersand replacement=\&]
	{\C^I \oplus \C^J} \& {\T^I \times \C^J} \& {\C^I \times\C^J} \& {\C^I \times \T^J} \\
	{\C^{I'}\oplus \C^{J'}} \& {\T^{I'}\times \C^{J'}} \& {\C^{I'}\times \C^{J'}} \& {\C^{I'}\times \T^{J'}}
	\arrow["{E_I}", from=1-1, to=1-2]
	\arrow["{\widetilde{L}}"', from=1-1, to=2-1]
	\arrow["{E'_{I'}}"', from=2-1, to=2-2]
	\arrow[hook, from=1-2, to=1-3]
	\arrow["{\overline{E}_J}", from=1-3, to=1-4]
	\arrow[dashed, from=1-4, to=2-4]
	\arrow["{\overline{E}'_{J'}}"', from=2-3, to=2-4]
	\arrow["{\overline{L}}"', from=1-3, to=2-3]
	\arrow[hook, from=2-2, to=2-3]
	\arrow["{\overline{L}}"', from=1-2, to=2-2]
\end{tikzcd}
\end{equation}
The morphism $\widetilde{L}$ descends to $\T^I \times \C^J$ because $\widetilde{L}(\widetilde{\sigma}) \subset \widetilde{\sigma}'$ but since we do not make enough restrictions on $J$ and $J'$, the morphism $\overline{L}$ has no reason to descend to a morphism $\C^I \times \T^J \to \C^{I'}\times \T^{J'}$ (like the simplicial case).

Let $T \in \Afrak$ and 
\[\begin{tikzcd}[ampersand replacement=\&]
	{\widetilde{T}^{cal}} \& {\C^I \times \T^J} \\
	T
	\arrow[from=1-1, to=2-1]
	\arrow["m^{cal}", from=1-1, to=1-2]
\end{tikzcd}\]
be an object of $\Uscr^{cal}_\sigma$ over $T$.

Let $\widehat{T}^{cal}$ be the fibre product 
\[\widetilde{T}^{cal} {}_m\!\times_{\widetilde{E}_J} (\C^{I} \oplus \C^J)=\left\{\left(\widetilde{t},z\right) \in \widetilde{T}^{cal} \times (\C^I \oplus \C^J) \mid m^{cal}\left(\widetilde{t}\right)=\overline{E}_J(z)\right\}
 .\]
The group $\Z^{N-k} \times E_I(\ker(h_\sigma))$ acts on $\widehat{T}^{cal}$ by 
\[ (p,E_I(w_1,w_2)) \cdot (\widetilde{t},z_1,z_2)=((pr_k (p),E(w_1,w_2)) \cdot \widetilde{t},z_1+w_1,E(w_2)z_2)
\]
where $pr_k: \Z^{N-k}=\Z^{N-d} \oplus \Z^{d-k} \to \Z^{N-d}$ is the projection (same definition for $pr_{k'}: \Z^{N'-k'} \to \Z^{N'-d'}$).

The linear map $H_{\chi\chi'}$ satisfies, for $i \in \{1,\ldots,k\}$, 
\[ H_{\chi\chi'}(e_i) \in \Z^{k'} \oplus 0
\]              
Then, the map $H_{\chi\chi'
}$ is of the form $\begin{pmatrix} * & * \\ 0 & M \end{pmatrix} \in \Mcal_{N,N'} (\Z)$ (we do not care of the upper part since we apply the map $E_I$).

In consequence, we can define the $\Z^{N'-d'} \times E'(\ker(h_{\sigma'\C}))$-principal bundle $\widetilde{T'}^{cal}$ as
\[\widetilde{T'}^{cal}\coloneqq \widehat{T}^{cal} \times_{(pr_{k'} \circ M) \times \overline{E}'_J \circ \overline{L}} (\Z^{N'-d'} \times E'(\ker(h_{\sigma'\C})))
\]
with associated equivariant map
\[\widetilde{m}^{cal}: \left(\widetilde{t},z,(q,E'(w))\right) \mapsto \overline{E}'_{J}\left(\overline{L}(z)\right) \cdot E'(\varphi'(0 \oplus q)+w)
\]
Hence, 
\[\begin{tikzcd}[ampersand replacement=\&]
	{\widetilde{T'}^{cal}} \& {\C^{I'} \times \T^{J'}} \\
	T
	\arrow[from=1-1, to=2-1]
	\arrow["{\widetilde{m}^{cal}}", from=1-1, to=1-2]
\end{tikzcd}\]
is an object of $\Uscr^{cal}_{\sigma'}$ over $T$.

\begin{Lemma}
The left square and the right square of the following diagram are cartesian: 
\[\begin{tikzcd}[ampersand replacement=\&]
	{\widehat{T}^{cal}} \& {\C^I \times\C^J} \& {\C^{I'} \times\C^{J'}} \&\& {\widehat{T}^{cal}\times_{M\times\overline{L}}(\Z^{N'-d'} \times E'_{I'}(\ker(h_{\sigma'\C})))} \\
	{\widetilde{T}^{cal}} \& {\C^I \times\T^J} \& {\C^{I'} \times\T^{J'}} \&\& {\widetilde{T}'^{cal}} \\
	\& T
	\arrow[from=1-1, to=1-2]
	\arrow["{\overline{L}}", from=1-2, to=1-3]
	\arrow["g"', from=1-5, to=1-3]
	\arrow["f", from=1-5, to=2-5]
	\arrow["{\widetilde{m}^{cal}}"', from=2-5, to=2-3]
	\arrow["{\overline{E}'_{J'}}", from=1-3, to=2-3]
	\arrow["{\overline{E}_J}", from=1-2, to=2-2]
	\arrow[from=1-1, to=2-1]
	\arrow["m^{cal}", from=2-1, to=2-2]
	\arrow[dashed, from=2-2, to=2-3]
	\arrow[from=2-1, to=3-2]
	\arrow[from=2-5, to=3-2]
\end{tikzcd}\]
where 
\[
f(\widetilde{t},z,q,E'_{I'}(t))=(\widetilde{t},z,pr_{k'}(q),E'(t))
\]
and
\[
g(\widetilde{t},z,q,E'_I(t))=E'_I(\varphi'(0 \oplus q)+t) \cdot \overline{L}(z)
\]
\end{Lemma}
\begin{proof}
See proof of lemma 5.3 of \citep{VQS}
\end{proof}

We define the image of a morphism between objects of $\Uscr_\sigma^{cal}$ by the same way as \citep{VQS} in diagram 6.10.

We get a stack morphism $\lscr^{cal}: \Uscr_{\sigma}^{cal} \to \Uscr_{\sigma'}^{cal}$.
\begin{Thm}\; \label{equiv_cat_aff}
\begin{itemize}
\item The stack morphism $\lscr^{cal}$ is a toric morphism.
\item Let $\sigma \subset \R^d$, $\sigma' \subset \R^{d'}$, $\sigma'' \subset \R^{d''}$ be cones and $(L,H)$, $(L',H')$ be calibrated quantum fans morphisms
between the calibrated quantum fans associated to it. Note $\lscr, \lscr'$ the toric morphisms associated to it and $\lscr''$ the toric morphism associated to $(L' \circ L,H' \circ H)$. Then, 
\begin{equation}
\lscr''=\lscr' \circ \lscr
\end{equation} 
\item Let $\lscr^{cal}$ be a torus morphism between $\Tscr^{cal}_{h,\Ical}$ and $\Tscr^{cal}_{h',\Ical'}$ and let $(L,H)$ be the induced linear morphisms. Then, $\lscr^{cal}$ extends as a toric morphism $\Uscr_{\sigma}^{cal} \to \Uscr_{\sigma'}^{cal}$ if, and only if, $(L,H)$ is a morphism of calibrated quantum fans.
\end{itemize}
\end{Thm}
 \begin{proof}
 
In \citep{VQS}, see the proof of lemma 5.4 for the second point and the proof of theorem 5.5 and theorem 6.2 for the first and third points.
 \end{proof}

The two first points tell us that we defined a functor between the full subcategory of the quantum fans given by a cone and the category of affine quantum toric varieties. This functor is an equivalence of category thanks to the third one.

Now, we can give the link between the construction of affine quantum variety of this paper and the construction of \citep{VQS}: 

\begin{Prop} \label{isom_simp_cal}
Suppose $\sigma$ simplicial. We have a toric isomorphism:
\[\Uscr_\sigma^{cal} \simeq Q^{cal}_{k,d,P_{\chi}^{-1}\circ \varphi}\] (in the same manner as paragraph 6.3 in \citep{VQS}).

\end{Prop}

\begin{proof}
If $\sigma$ is simplicial then $h_{\sigma\C}$ is a monomorphism and $\overline{h}_{\sigma}$ is an isomorphism. The morphism $(\overline{h}_{\sigma} P_\chi,P_\chi)$ is an isomorphism of calibrated quantum fans between the fan induced by the cone $C_{k,d}=\Cone(e_1,\ldots,e_k)$ and the calibration $P_{\chi}^{-1}\circ \varphi: \Z^N \to P_{\chi^{-1}}(\overline{h}_\sigma^{-1}(\Gamma))$ and the fan induced by $\sigma$ and the calibration $h$: \\ Firstly, the following diagram commute: 
\[\begin{tikzcd}[ampersand replacement=\&]
	{\C^N} \& {\C^I} \\
	{\C^N} \& {\C^d}
	\arrow["h"', from=2-1, to=2-2]
	\arrow["{P_\chi}"', from=1-1, to=2-1]
	\arrow["{P_\chi^{-1}\varphi}", from=1-1, to=1-2]
	\arrow["{\overline{h}_{\sigma\C}P_\chi}", from=1-2, to=2-2]
\end{tikzcd}\]
since $\varphi=\overline{h}_{\sigma\C}^{-1} \circ h \circ P_\chi$. The map $\overline{h}_{\sigma\C} P_\chi$ sends $C_{k,d}$ onto $\sigma$ since 
\[\overline{h}_{\sigma\C} P_\chi(e_i)=\overline{h}_{\sigma\C} (e_{\chi(i)})=v_{\chi(i)} \in \sigma \]
The other points are obvious.

Hence, it induces a toric isomorphism (thanks to theorem \ref{equiv_cat_aff})
\[
\Uscr_\sigma^{cal} \simeq Q^{cal}_{k,d,P_{\chi}^{-1}\circ \varphi}
\]
\end{proof}
\subsection{Quantum toric varieties}
\label{Glueing}

\subsubsection{Definition}
Let $\sigma=\sigma_I=\Cone(v_i, i \in I)$ and $\tau=\sigma_J=\Cone(v_j, j \in J)$ be two cones of $\Delta$ with a non-empty intersection. Note $\widetilde{\sigma}$(resp. $\widetilde{\tau}$) the associated cone (see \eqref{cone_Rr}) of $\R^I$ (resp. $\R^J$), $h_\sigma,h_\tau$ the associated group homomorphisms. In this subsection, we will see how to glue the quantum toric varieties associated to them.  

We will suppose that $\sigma$ and $\tau$ are of dimension $d$ but the following construction works for cones of any dimension.

Let $H_{\sigma\tau}: \C^{I \cup J} \to \C^d$ be the linear map such that ${H_{\sigma\tau}}_{|\C^I}=h_{\sigma\C}$ and ${H_{\sigma\tau}}_{|\C^J}=h_{\tau\C}$.
The identity morphism $\Tscr^{cal}_{d,\Ical} \to \Tscr^{cal}_{d,\Ical}$ induces torus isomorphisms \begin{equation} \label{isom_1}
[\T^{I}/\Z^{N-d} \times E(\ker(h_{\sigma\C}))] \simeq [\T^{I \cup J}/\Z^{N-d} \times E(\ker(H_{\sigma\tau\C}))]
\end{equation}

\begin{equation} \label{isom_2}
[\T^{I}/\Z^{N-d} \times \ker(h_{\tau\C})] \simeq [\T^{I \cup J}/\Z^{N-d} \times E(\ker(H_{\sigma\tau\C}))]
\end{equation}

Moreover, these morphisms extend to toric isomorphisms between $\Uscr_\sigma^{cal}=[\C^I/\Z^{N-d} \times E(\ker(h_{\sigma\C})]$ and $[\C^I \times \T^J/\Z^{N-d}\times E(\ker(H_{\sigma\tau}))]$ and between $\Uscr_\tau^{cal}=[\C^I/\Z^{N-d} \times E(\ker(h_{\tau\C}))]$ and $[\T^I \times \C^J/\Z^{N-d}\times E(\ker(H_{\sigma\tau}))]$ in the same way as proposition \ref{souscone}.

Note $K_I\coloneqq (I \cup J) \setminus I$ and $K_J\coloneqq (I \cup J) \setminus J$. \\
As the intersection $\sigma \cap \tau$ is non-empty, the intersection $\{0\} \times \widetilde{\tau} \cap \widetilde{\sigma} \times \{0\}$ in $\R^{I \cup J}$ is non-empty too. Hence, the classical theory gives us an open toric subvariety $\Sscr_{\sigma\tau}$ of $U_{\{0\} \times \widetilde{\tau}}$, $\Sscr_{\tau\sigma}$ an open toric subvariety of $U_{\widetilde{\sigma} \times \{0\}}$ and an toric isomorphism $\varphi: \Sscr_{\sigma\tau} \to \Sscr_{\tau\sigma}$. With some computations, we see that (with the decomposition $\C^{I \cup J}=\C^{K_J} \oplus \C^{I \cap J} \oplus \C^{K_I}$):
\begin{equation}U_{\widetilde{\sigma} \times \{0\}} =\C^I \times \T^{K_I}, U_{\{0\} \times \widetilde{\tau}}=\T^{K_J} \times \C^J\subset \C^{I \cup J}  \end{equation}
and
\begin{equation} \label{gluing_open}
\varphi=id \colon \Sscr_{\tau\sigma}=\T^{K_J} \times \C^{I \cap J} \times \T^{K_I} \to \Sscr_{\sigma\tau}=\T^{K_J} \times \C^{I \cap J} \times \T^{K_I} 
\end{equation}

The identity descends to a stack isomorphism (thanks to the linear isomorphism induced by the permutation $\chi'^{-1}\chi$ and the theorem \ref{equiv_cat_aff})
\[[\Sscr_{\tau\sigma}/\Z^{N-d} \times E(\ker(H_{\sigma\tau}))] \simeq [\Sscr_{\sigma\tau}/\Z^{N-d} \times E(\ker(H_{\sigma\tau}))] \] 
(since the toric varieties $\Sscr_{\sigma\tau}$, $\Sscr_{\tau\sigma}$ are preserved by the actions).

 Now, thanks to the equality in \eqref{gluing_open} and the isomorphisms \eqref{isom_1} and \eqref{isom_2}, we get a toric isomorphism $\gscr_{IJ}$ between $\Uscr_{\tau\sigma}^{cal}\coloneqq [\T^{K_J} \times \C^{I \cap J}/\Z^{N-d} \times E(\ker(h_{\sigma\C}))] \hookrightarrow [\C^I/\Z^{N-d} \times E(\ker(h_{\sigma\C})] $ and $\Uscr_{\sigma\tau}^{cal}\coloneqq [\C^{I \cap J} \times \T^{K_I}/\Z^{N-d} \times E(\ker(h_{\tau\C})] \hookrightarrow [\C^J/\Z^{N-d} \times E(\ker(h_{\tau\C})] $.

\begin{Rem}
This transitions maps verify a cocycle condition since the identity map does.
\end{Rem}

With the previous discussion, we can define quantum toric varieties: 
\begin{Def} \label{QTV_cal}
Let $T \in \Afrak$. An object of $\Xscr^{cal}_{\Delta,h^{cal},\Ical}$ over $T$ is a covering $(T_I\coloneqq T_{\sigma_I})$ of $T$ indexed by the set of maximal cones $\Ical_{max}$ together with an object 
\[\begin{tikzcd}[ampersand replacement=\&]
	{\widetilde{T}_I} \& {\C^I \times \T^K} \\
	{T_I}
	\arrow[from=1-1, to=2-1]
	\arrow["{m_I}", from=1-1, to=1-2]
\end{tikzcd}\]

of $[\C^I \times \T^K/\Z^{N-d} \times E(\ker(\overline{h}_{\sigma_I\C}))](T_I)$ for every $\sigma_I \in \Ical_{max}$, satisfying for any couple $(I,I')$ with non-empty intersection $J$

\[ \gscr_{II'}\begin{pmatrix}\begin{tikzcd}[ampersand replacement=\&]
	{\widetilde{T}_{I \supset J}} \& {\Sscr_{\sigma_I\sigma_{I'}}} \\
	{T_I}
	\arrow[from=1-1, to=2-1]
	\arrow["{m_I}", from=1-1, to=1-2]
\end{tikzcd}\end{pmatrix}=\begin{matrix}\begin{tikzcd}[ampersand replacement=\&]
	{\widetilde{T}_{I' \supset J}} \& {\Sscr_{\sigma_{I'}\sigma_{I}}} \\
	{T_{I'}}
	\arrow[from=1-1, to=2-1]
	\arrow["{m_{I'}}", from=1-1, to=1-2]
\end{tikzcd}\end{matrix}\]

where $\widetilde{T}_{I' \supset J}\coloneqq m_I^{-1}(\Sscr_{\sigma_I \sigma_{I'}})$ and $\widetilde{T}_{I' \supset J}\coloneqq m_{I'}^{-1}(\Sscr_{\sigma_{I'} \sigma_I})$ 

A morphism of $\Xscr_{\Delta,h^{cal},\Ical}$ over a toric morphism $T \to S$ is defined as in \citep{VQS} with the necessary modifications.

\end{Def}

\subsubsection{Toric morphisms}
A morphism of quantum toric varieties $\Xscr^{cal}_{\Delta,h^{cal},\Ical} \to \Xscr^{cal}_{\Delta',(h^{cal})',\Ical'} $  is a collection of compatible toric morphisms between their affine pieces. In other words,   

\begin{Def}
A morphism $\Xscr_{\Delta,h^{cal},\Ical}^{cal} \to \Xscr_{\Delta',\left(h^{cal}\right)',\Ical'}^{cal}$ of calibrated quantum toric varieties is a collection of toric morphisms $\lscr_{\sigma}\colon\Uscr_{\sigma}^{cal} \to \Uscr_{\sigma'}^{cal}$ for all maximal cones of $\Delta$ compatible with the gluing i.e. if the intersection of $\sigma$ and $\tau$ is not empty then the morphism $\lscr_\sigma$ (resp. $\lscr_{\tau}$) restricts to a morphism $\Uscr_{\tau\sigma}^{cal} \to \Uscr_{\tau'\sigma'
}^{cal}$ (resp. $\Uscr_{\sigma\tau}^{cal} \to \Uscr_{\sigma'\tau'}^{cal}$) and such that the following equality holds on $\Uscr_{\tau\sigma}^{cal}$
\begin{equation} \label{toric_morphism _equality}
\gscr'_{\sigma'\tau'} \lscr_{\sigma}=\lscr_{\tau} \gscr_{\sigma\tau}
\end{equation}
 for each cones $\sigma,\tau$ with non-empty intersection.                                                                                                       
\end{Def}

In paragraph \ref{affine_morphisms}, we proved a correspondence between the affine quantum toric varieties morphisms and the calibrated quantum morphisms between cones. We will complete it: 

Let $\Xscr_{\Delta,h^{cal},\Ical}^{cal} \to \Xscr_{\Delta',\left(h^{cal}\right)',\Ical'}^{cal}$ be a morphism of calibrated quantum toric varieties i.e. a collection of toric morphisms $\lscr_{\sigma}\colon\Uscr_{\sigma}^{cal} \to \Uscr_{\sigma'}^{cal}$ respecting the equalities \eqref{toric_morphism _equality}
Let $(L_{\sigma},H_{\sigma})$ be the linear maps associated to $\lscr_\sigma$, $(g_{\sigma\tau},k_{\sigma\tau})$ the linear isomorphisms associated to $\gscr_{\sigma\tau}$ and $(g_{\sigma'\tau'},k'_{\sigma'\tau'})$ the linear isomorphisms associated to $\gscr'_{\sigma'\tau'}$. Then, thanks to the second point of theorem \ref{equiv_cat_aff}, the equations \eqref{toric_morphism _equality} becomes:
\begin{equation}
g'_{\sigma'\tau'} \circ H_\sigma=H_{\tau}\circ g_{\sigma\tau} \text{ and } k'_{\sigma'\tau'} \circ L_\sigma=L_{\tau}\circ k_{\sigma\tau}
\end{equation}
Hence, we can glue these calibrated quantum fans morphisms into one between $(\Delta,h^{cal},\Ical) \to (\Delta',\left(h^{cal}\right)',\Ical')$. \\
Conversely, a calibrated quantum fans morphism $(\Delta,h^{cal},\Ical) \to (\Delta',\left(h^{cal}\right)',\Ical')$ defines toric morphism between affine calibrated quantum toric varieties verifying \eqref{toric_morphism _equality} i.e. defines a toric morphism $\Xscr_{\Delta,h^{cal},\Ical}^{cal} \to \Xscr_{\Delta',\left(h^{cal}\right)',\Ical'}^{cal
}$ \\
We proved our main theorem: 
\begin{Thm}
The correspondence $(\Delta,h^{cal},\Ical) \to \Xscr^{cal}_{\Delta,h^{cal},\Ical }$ is functorial and defines an equivalence of categories between the category of calibrated quantum fan and the category of calibrated quantum toric varieties.

\end{Thm}

\renewcommand{\theThm}{\arabic{section}.\arabic{Thm}}
\section{GIT-like construction}
\label{GIT}
In this section, we will discuss the realization of a quantum toric variety as a global quotient stack.

We can build the gluing of quantum toric varieties in another way than the previous section: \\
Let $\Sscr$ be the gluing of the toric varieties $\C^I$ and $\C^J$ along the intersection $\{0\} \times \widetilde{\tau} \cap \widetilde{\sigma} \times \{0\}$. 
Then, we define the gluing of $\Uscr^{cal}_{\sigma}$ and $\Uscr^{cal}_{\tau}$ is the quotient stack $[\Sscr/\Z^{N-d} \times E(\ker(H_{\sigma\tau}))]$. Then, we deduce 

\begin{Thm} \label{incomplete_GIT}
Let $A$ be the set $A\coloneqq \bigcup_{\sigma_I \in \Delta} I$ (as defined in \ref{calib_q_fan_def}) . Let $\widetilde{A}$ a subset of $\{1,\ldots,N\} \setminus \Ical$ such that $(v_i,i \in \widetilde{A})$ is free and such that $\C^d=\Vect(v_i,i \in A) \oplus \Vect(v_i,i \in \widetilde{A})$
Let $h_A: \Z^{A \cup \widetilde{A}} \to \Gamma$ be  the group homomorphism defined by $h(e_i)=v_i$ for $i \in A \cup \widetilde{A}  $. Note $\Sscr_A$ the toric variety given by the associated fan of $\Delta$ in $\R^A$ given by
\[ \Delta_h \cap \R^A=\{ \tau\mid \exists \sigma \in \Delta, \tau \preceq \widehat{\sigma} \cap \R^A \}.\]
We define an action of $\Z^{N-d} \times E(\ker(h_A))$ on $\Sscr_A\times \T^{\widetilde{A}}$ in the same way as \ref{cone_non_max}. Then, 
\[\Xscr^{cal}_{\Delta,h^{cal},\Ical} \simeq [\Sscr_A \times \T^{\widetilde{A}}/\Z^{N-d} \times E(\ker(h_{A\C}))]\]
as stacks
\end{Thm}

Moreover, in the same way as in the beginning of subsection \ref{Glueing}, we have a toric isomorphism
\[
[\C^I \times \T^J/\Z^{N-d} \times \ker(h_{\sigma \C})] \simeq [\C^I \times \T^{I^c}/\Z^{N-d} \times \ker(h^{cal}_{\C})] 
\]
Hence, we get a GIT-like realization of $\Xscr^{cal}_{\Delta,h^{cal},\Ical}$:
\begin{Thm} \label{GIT_noGale}
Let $\Sscr$ be the toric variety associated to the associated fan $\Delta_h$. Then, we have a stack isomorphism 
\begin{equation} \label{isom_GIT_noGale}
    \Xscr^{cal}_{\Delta,h^{cal},\Ical} \simeq [\Sscr/\Z^{N-d} \times E(\ker(h^{cal}_\C))]
\end{equation}
which restricts to a torus isomorphism between the associated quantum torus on each affine chart.
\end{Thm}

By contrast to the simplicial case, we cannot realize $\Xscr^{cal}_{\Delta,h,\Ical}$ as a quotient of a toric variety by an action of $\C^{N-d}$ via Gale transform (the "quantum GIT" of \citep{VQS}).

Indeed, let $k: \C^{N-d} \to \C^N$ be the map induced by a Gale transform (see \citep{CP} for the details) of the family $(h^{cal}(e_i), i \in [\![1,N]\!])$ i.e. a map such that the sequence 
\[\begin{tikzcd}[ampersand replacement=\&]
	0 \& {\C^{N-d}} \& {\C^N} \& {\C^d} \& 0
	\arrow[from=1-1, to=1-2]
	\arrow["k", from=1-2, to=1-3]
	\arrow["{h_\C^{cal}}", from=1-3, to=1-4]
	\arrow[from=1-4, to=1-5]
\end{tikzcd}\]
is exact. 

We would like to prove that $\Xscr^{cal}_{\Delta,h^{cal},\Ical}$ and $[\Sscr/\C^{N-d}]$ are isomorphic (where $\C^{N-d}$ acts on $\Sscr$ through $E \circ k$) in order to generalize the theorem 7.6 of \citep{VQS}). 

For each $\sigma \in \Delta$, the corresponding open substack of $\Xscr^{cal}_{\Delta,h^{cal},\Ical}$ is $\Uscr_\sigma^{cal}=[\C^I \times \T^J/\Z^{N-d} \times E(\ker(h_{\sigma\C}))]$ and the substack of $[\Sscr/\C^{N-d}]$ is $\Uscr'_\sigma\coloneqq [\C^I \times \T^{I^c}/\C^{N-d}]$ where $I^c=\{1,\ldots,N\} \setminus I$. 
\begin{Prop} \label{noGale} 
The stacks $\Uscr_{\sigma}^{cal}$ and $\Uscr'_\sigma$ are not isomorphic if $\sigma$ is not simplicial.
\end{Prop} 

\begin{proof}
The groupoid associated to $\Uscr_\sigma^{cal}$ and the groupoid associated to $\Uscr'_\sigma$ cannot be Morita-equivalent since their isotropy groups (i.e. the stabilizer on a point by the action) are not isomorphic (see theorem 4.4 of \citep{Morita}).

More precisely, the stabilizer of the action of $\Z^{N-d} \times E(\ker(h_{\sigma\C}))$ at each point of $(\C^{\widetilde{I}} \oplus 0) \times \T^J$ is $E(\ker(h_{\sigma\C})) \times \ker(h^{cal})$ and there is no point of $\C^I \times \T^{I^c}$ with an isomorphic stabilizer group for the $\C^{N-d}$-action. Indeed, the different stabilizer groups for the $\C^{N-d}$-action are $k^{-1}(\ker(h_{\sigma\C})\cap \C^K+\ker(h^{cal}))$ at $\C^{\widetilde{I}} \oplus (\ker(h_{\sigma\C})  \cap \{0\}^{K}) \times \T^{I^c}$ for $\emptyset \neq K \varsubsetneq I$ and $k^{-1}(\ker(h_{\sigma\C})+\ker(h^{cal}))$ at $(\C^{\widetilde{I}} \oplus \{0\}) \times \T^{I^c}$. These groups cannot be isomorphic to $E(\ker(h_\sigma)) \times \ker(h^{cal})$(think of the lack of isomorphism between $\C^n$ and $(\C^*)^n \times \Z^n$ due to the torsion).
\end{proof}


\begin{Ex}
Let $\varepsilon \in \R_{>0}$. Let $v=\left(-\frac{1}{\varepsilon},\frac{2+2\varepsilon}{\varepsilon},-\frac{2}{\varepsilon}\right)$ be a vector of $\R^3$, $\Gamma=\Z^3+\Z v$ be a subgroup of $\R^3$, $h^{cal}: \Z^6 \to \Gamma$ be a standard calibration of $\Gamma$ such that $h^{cal}(e_4)=-e_1$, $h^{cal}(e_5)=-e_2$, $h^{cal}(e_6)=v$ and $\Delta$ the fan of $\R^3$ whose maximal fans are 
\begin{align*}
    \Delta_{max}=&\left\{\Cone(e_1,\pm e_2,e_3),\Cone(-e_1,-e_2,e_3),\right.\\
    &\left.\Cone(e_1,\pm e_2,v),\Cone(-e_1,-e_2,v),\Cone(-e_1,e_2,e_3,v) 
    \right\}
\end{align*}

This fan is the normal fan associated to (a deformation of) the non-simple slice of the flop of type (2,2) of the cube (see figure 5 of \citep{bosio2006} for the description of the flop).
\begin{figure}[h]
\centering
\caption{The polytope associated to the fan $\Delta$}
\begin{tikzpicture}[scale=0.7]
    \draw (1,-2) to (1,1) to (0,2)node[right]{$e_3$} to (-1,1)  to (0,0.25);
    \draw[dashed] (-1,1) node[left]{$-e_2$}to (-0.125,1.5)  to (1,1) node[right]{$e_2$} ;
    \draw[dashed] (0,2) to (-0.125,1.5) to (1,-2) ; 
    \draw (0,2) to (0,0.25) to (1,-2) node[right]{$v$}  to (-1,1);
\end{tikzpicture}
\end{figure}
There are six simplicial maximal cones and one maximal non-simplicial cone in $\Delta$. We associate to each of them an affine quantum toric variety (like \ref{ex_max}) and define how to glue them.
Let's see how to describe it globally: \\
The toric variety $\Sscr$ associated to $\Delta_{h^{cal}}$ is
\[
\Sscr=(\C^2 \setminus \{0\})^3 \setminus [\C^* \times \C^3 \times(\C^*)^2 \cup \C^* \times \C \times \C^* \times \C \times \C^* \times \C] \cup (\C^* \times \C^3 \times \C^* \times \C )
\]
and the kernel of $h_\C^{cal}$ is 
\[
\ker(h_\C^{cal})=\C(1,0,0,1,0,0) \oplus \C(0,1,0,0,1,0) \oplus \C\left(-\frac{1}{\varepsilon},\frac{2+2\varepsilon}{\varepsilon},-\frac{2}{\varepsilon},0,0,-1\right)
\]
The quantum toric variety $\Xscr_{\Delta,h^{cal},\emptyset}$ is isomorphic to the quotient stack $[\Sscr/\Z^{n-d} \times E(\ker(h^{cal}))]$ where $\Z^{n-d} \times E(\ker(h^{cal}))$ acts on $\Sscr$ through
\[
(p,E(t))\cdot z=E(h(p)+t)z
\]\end{Ex}
\renewcommand{\theThm}{\arabic{section}.\arabic{Thm}}
\section{Forgetting calibration and gerbe structure}
\label{forget_cal}
We can associate to each affine quantum toric variety $\Uscr_\sigma^{cal}=[\C^I \times \T^J/\Z^{N-d} \times E(\ker(h_{\sigma\C}))]$ a "non-calibrated quantum toric variety" $\Uscr_\sigma\coloneqq [\C^I \times \T^J/E(h_{\sigma\C}^{-1}(\Gamma))]$. \\
More precisely, note $\Xi$ be the kernel of $h^{cal}$ and let $T$ be an object of $\Afrak$ and $(\widetilde{T}^{cal},m^{cal})$ be an object of $\Uscr_\sigma^{cal}(T)$. \\
Since \[h \times id \colon \Z^{N-d} \times E(\ker(h_{\sigma\C})) \to \frac{\Z^{N-d}\times E(\ker(h_{\sigma\C}))}{\Xi \times 0}=E(\Gamma) \times E(\ker(h_{\sigma\C}))=E(h_{\sigma\C}^{-1}(\Gamma))\] is a $\Xi$-covering then $\widetilde{T}=\widetilde{T}^{cal}/\Xi \to T$ is a $E(h_{\sigma\C}^{-1}(\Gamma))$-principal bundle. \\
Note $m: [\widetilde{t}^{cal}]\in \widetilde{T} \mapsto m^{cal}(\widetilde{t}^{cal}) \in \C^I \times \T^J$. This map is well-defined: 
\[ m^{cal}((\xi,0) \cdot \widetilde{t}^{cal})=E(\varphi(\xi,0))m^{cal}(t^{cal})=m^{cal}(t^{cal}) \]

In other words, a non-calibrated affine quantum toric varieties is the stack obtained by forgetting the ineffectivity of the action of $\Z^{N-d} \times E(\ker(h_{\sigma\C}))$ on $\Sscr$. 

\begin{Def}
The non-calibrated quantum toric variety associated to $\sigma_I$ and the group $\Gamma$ is the stack 
\[
\Uscr_\sigma = [\C^I \times \T^J/E(h_{\sigma\C}^{-1}(\Gamma))]
\]
where $|J|=n-\dim(\sigma)$
\end{Def}

The expression "non-calibrated" comes from the simplicial case where the authors of \cite{VQS} use only the generators of the 1-cones and not a calibration of the group. It is not used for the non-simplicial case due to the lack of compatibility with morphisms :

We can define morphisms of non-calibrated affine quantum toric varieties in the same way as \ref{morphism_cal}:
\begin{Def} \label{morphism_uncal}
A toric morphism between two non-calibrated affine quantum toric varieties $\Uscr_{\sigma_1}=[\C^I \times \T^J/G_1]$ and $\Uscr_{\sigma_2}=[\C^{I'} \times \T^{J'}/G_2]$ is a stack morphism $[\C^I \times \T^J/G_1] \to [\C^{I'} \times \T^{J'}/G_2]$ which restricts to a presented non-calibrated torus morphism $[\T^{I} \times \T^J/G_1] \to [\T^{I'} \times \T^{J'}/G_2]$
\end{Def}

We can see that a morphism of affine quantum toric varieties descends to quotient and induces a morphism of non-calibrated affine quantum toric varieties. Hence, we have defined a functor $\mathrm{f}: \Uscr_\sigma^{cal} \mapsto \Uscr_\sigma$ from the category of affine quantum toric varieties to the category of non-calibrated affine quantum toric varieties.

\begin{Rem} \label{non_cal_morph}
We can follow the proof of \ref{equiv_cat_aff} in order to see that a morphism $(L,H)$ of calibrated quantum fan morphism (we need the two morphisms due to the presence of the kernel of $h_{\sigma\C}$) leads to a morphism $\lscr$ of non-calibrated affine quantum toric varieties. 
\end{Rem}

\begin{Lemma} \label{forgetful_functor_f}
This functor coincide with the functor $\mathrm{f}$ of \citep{VQS} (section 6.2) on simplicial quantum toric variety.
\end{Lemma}

Moreover, the functor $\mathrm{f}$ induces (with lemma \ref{equiv_6uple}) a functor $\widetilde{\mathrm{f}}$ on the category of presented quantum tori defined by  
\[
\widetilde{\mathrm{f}}(\Tscr^{cal}_{h,\Ical},h': \Z^N \to G,\Ical',L,H,s)= (\Tscr_{d,\Gamma},L)
\]
and 
\[\widetilde{\mathrm{f}}(\Lcal,\Hcal,\Scal,\Lcal',\Hcal',\Scal')=(\Lcal,\Lcal')\]

\begin{Prop}
We have a commutative diagram
\[\begin{tikzcd}[ampersand replacement=\&]
	{\Uscr^{cal}_\sigma} \& {\Uscr^{cal}_{\sigma'}} \\
	{\Uscr_\sigma} \& {\Uscr_\sigma}
	\arrow["{\mathrm{f}}"', from=1-1, to=2-1]
	\arrow["{\lscr^{cal}}", from=1-1, to=1-2]
	\arrow["{\mathrm{f}}", from=1-2, to=2-2]
	\arrow["\lscr"', from=2-1, to=2-2]
\end{tikzcd}\]
\end{Prop}

We can adapt the definition of quantum toric varieties to the non-calibrated case:

\begin{Def} \label{QTV_uncal}
Let $T \in \Afrak$. An object of $\Xscr_{\Delta,\Gamma}$ over $T$ is a covering $(T_I\coloneqq T_{\sigma_I})$ of $T$ indexed by the set of maximal cones $\Ical_{max}$ together with the image by $\mathrm{f}$ of an object

\[\begin{tikzcd}[ampersand replacement=\&]
	{\widetilde{T}_I} \& {\C^I \times \T^K} \\
	{T_I}
	\arrow[from=1-1, to=2-1]
	\arrow["{m_I}", from=1-1, to=1-2]
\end{tikzcd}\]

of $[\C^I \times \T^K/\Z^{N-d} \times E(\ker(\overline{h}_{\sigma_I\C})](T_I)$ for every $\sigma_I \in \Ical_{max}$, satisfying for any couple $(I,I')$ with non-empty intersection $J$
\begin{equation} \label{noncal_map_change}
     \mathrm{f}\left(\gscr_{II'}\begin{pmatrix}\begin{tikzcd}[ampersand replacement=\&]
	{\widetilde{T}_{I \supset J}} \& {\Sscr_{\sigma_I\sigma_{I'}}} \\
	{T_I}
	\arrow[from=1-1, to=2-1]
	\arrow["{m_I}", from=1-1, to=1-2]
\end{tikzcd}\end{pmatrix}\right)=\mathrm{f}\begin{pmatrix}\begin{tikzcd}[ampersand replacement=\&]
	{\widetilde{T}_{I' \supset J}} \& {\Sscr_{\sigma_{I'}\sigma_I}} \\
	{T_{I'}}
	\arrow[from=1-1, to=2-1]
	\arrow["{m_{I'}}", from=1-1, to=1-2]
\end{tikzcd}\end{pmatrix}
\end{equation}
where $\widetilde{T}_{I' \supset J}\coloneqq m_I^{-1}(\Sscr_{\sigma_I \sigma_{I'}})$ and $\widetilde{T}_{I' \supset J}\coloneqq m_{I'}^{-1}(\Sscr_{\sigma_{I'} \sigma_I})$ 

A morphism of $\Xscr_{\Delta,\Gamma}$ over a toric morphism $T \to S$ is defined by applying $\mathrm{f}$ to a morphism of $\Xscr^{cal}_{\Delta,h^{cal},\Ical}$ over this morphism.
\end{Def}

The definition of non-calibrated quantum toric varieties uses the forgetful functor $\mathrm{f}$ because, as remarked previously (see remark \ref{non_cal_morph}), we need more data than a non-calibrated quantum fan in order to defining transitions maps. Hence, these varieties can not be defined without their calibrated counterpart.
It is why the (non-simplicial) non-calibrated quantum toric varieties are not studied in this paper and thus we do not investigate adaptations of theorems of \cite{VQS} regarding non-calibrated quantum toric varieties (cf. theorem 5.18 or 7.10).

In the same way as proposition 6.20 of \citep{VQS}, we have :
\begin{Thm}
The quantum toric variety $\Xscr^{cal}_{\Delta,h,\Ical}$ is a gerbe over $\Xscr_{\Delta,\Gamma}$ with band $\Z^{\mathrm{rk}(\Xi)}$. In particular, if $\Xi=0$ then these two stacks are isomorphic. 
\end{Thm}

This structure of gerbe induces a $\T^{\mathrm{rk}(\Xi)}$-bundle over $\Xscr_{\Delta,\Gamma}$ up to homotopy. The end of this section will be devoted to describing it:

Let $\sigma=\sigma_I \in \Delta$ be a maximal cone, $h_{\sigma\C}$ be the $\C$-linear morphism associated to it, $\psi: \C^I/\ker(h_{\sigma\C}) \to \C^d$ induced by $h_{\sigma\C}$. \\
Let $\widetilde{I}$ be a subset of $I$ such that $(h(e_i), i\in \widetilde{I})$ is a basis of $\Vect(\sigma)$, $J$ be a set of cardinal $d-\dim(\sigma)$ such that \[\Vect(h^{cal}(e_i),i \in I \cup J)=\C^d,\] let $\chi \in \Sfrak_N$ be a permutation such that
\begin{equation} \label{condition_permutation}
 \chi(\{1,\ldots,\dim(\sigma)\})=\widetilde{I} \text{ and } \chi(\{\dim(\sigma)+1,\ldots,d\})=J 
\end{equation}
and $P_\chi \in \GL_N(\R)$ be the map associated to $\chi$.

Let $\varphi=(\varphi_1,\varphi_2)$ be the linear map $\psi^{-1} h P_\chi: \C^N \to \C^I/\ker(h_\sigma) \times \C^J$. Thanks to the conditions \eqref{condition_permutation}, we get the following diagram
\[\begin{tikzcd}[ampersand replacement=\&]
	{\C^N} \&\& {\C^I \times \C^J/\ker(h_\sigma)} \\
	{\T^I/E(\ker(h_{\sigma\C})) \times \C^J \times \C^{N-d} } \&\& {\T^I \times \C^J/\ker(h_\sigma)} \\
	{[\C^I/E(\ker(h_{\sigma\C})) \times \C^J \times \C^{N-d} ]} \&\& {[\C^I \times \C^J/E(\ker(h_\sigma))]} \\
	{[\C^I/E(\ker(h_{\sigma\C})) \times \T^J \times \C^{N-d} ]} \&\& {[\C^I \times \T^J/E(\ker(h_\sigma))]} \\
	{[\C^I/E(\ker(h_{\sigma\C})) \times \T^J \times \T^{N-d} ]} \&\& {\Uscr_{\sigma}^{cal}}
	\arrow["\varphi", from=1-1, to=1-3]
	\arrow["{(\psi^{-1},id_{\C^{N-d}}) \circ E_k}"', from=1-1, to=2-1]
	\arrow[hook', from=2-1, to=3-1]
	\arrow["{(id,E,id)}"', from=3-1, to=4-1]
	\arrow["{\overline{E_d}}"', from=4-1, to=5-1]
	\arrow["{E_I}", from=1-3, to=2-3]
	\arrow[hook', from=2-3, to=3-3]
	\arrow["{\overline{E}_J}", from=3-3, to=4-3]
	\arrow[from=4-3, to=5-3]
	\arrow["{\varphi^{cal}}", from=5-1, to=5-3]
	\arrow["{\overline{\varphi}}", from=2-1, to=2-3]
	\arrow["{\overline{\varphi}}", from=3-1, to=3-3]
	\arrow["{\widehat{\varphi}}", from=4-1, to=4-3]
\end{tikzcd}\]
where for all $(z_1,z_2) \in \T^I \times \C^J$, $w \in \C^{N-d}$ 
\[\overline{\varphi}([z_1],z_2,w)=([E(\varphi_1(0 \oplus w))z_1], \varphi_2(0 \oplus w)+z_2) \]
and 
$z \in \C^I \times \T^J$, $w \in \C^{N-d}$,
\[\widehat{\varphi}([z_1],z_2,w)=([E(\varphi(0 \oplus w))z] \]
(we can translate this equality in terms of principal bundles)

The group $\Z^{N-d}$ acts on $\left[\C^I/E(\ker(h_{\sigma\C})) \right] \times \T^J \times \C^{N-d}$ by translation on the last factor and on $\left[\C^I/E(\ker(h_{\sigma\C})) \right] \times \T^J \times \C^{N-d}$ through the morphism $\varphi$. The morphism $\widehat{\varphi}$ is equivariant for these two actions. Hence, $\widehat{\varphi}$ descends to a stack morphism $\left[\C^I/E(\ker(h_{\sigma\C}) \right] \times \T^J \times \T^{N-d} \to \Uscr^{cal}_\sigma  $

In the same manner, the projection on the first factor \[\left([\C^I/E(\ker(h_{\sigma\C}))] \times \T^J\right) \times \C^{N-d} \to [\C^I/E(\ker(h_{\sigma\C}))] \times \T^J \] is equivariant for the action (on the source) of $\Z^{N-d}$ defined by, for $p \in \Z^{N-d}$ ,  $(z,w) \in \C^I \times \T^J$
\begin{equation} \label{action_ZN-d}
p \cdot (z,w)=(E(\varphi(0 \oplus p))z, w+p )  
\end{equation}
(it is well-defined since the actions are multiplicative) and the action of $\Z^{N-d}$ on the target defined by, for $p \in \Z^{N-d}$ and for $(z,w) \in \C^I \times \T^J $
\begin{equation}
p \cdot z=E(\varphi(0 \oplus p))z  
\end{equation}
Hence, the projection descends to a stack morphism:

\begin{equation} \label{diag_triv_1}
\begin{tikzcd}[ampersand replacement=\&]
	{([\C^I/E(\ker(h_{\sigma\C}))] \times \T^J)\times \C^{N-d}} \&\& {([\C^I/E(\ker(h_{\sigma\C}))] \times \T^J)} \\
	{[(\C^I\times \T^J)\times \C^{N-d}/E(\ker(h_{\sigma\C}))\times \Z^{N-d}] } \& {} \& {\Uscr^{cal}_\sigma}
	\arrow[from=1-1, to=2-1]
	\arrow["{\pi_1}", from=1-1, to=1-3]
	\arrow[from=1-3, to=2-3]
	\arrow["\pscr", from=2-1, to=2-3]
\end{tikzcd}
\end{equation}
 
Moreover, we have the following commutative diagram
\[\begin{tikzcd}[ampersand replacement=\&]
	{([\C^I/E(\ker(h_{\sigma\C}))] \times\T^J)\times\C^{N-d}} \&\& {[\C^I\times\T^J/E(\ker(h_{\sigma\C}))] } \\
	{[\C^I/E(\ker(h_{\sigma\C}))] \times\T^J\times\C^{N-d}} 
	\arrow["{\pi_1}", from=1-1, to=1-3]
	\arrow["{(\widehat{\varphi},id)^{-1}}"', from=1-1, to=2-1]
	\arrow["{\widehat{\varphi}}"', from=2-1, to=1-3]
\end{tikzcd}\]
Since the morphism $(\widehat{\varphi},id)^{-1}$ is equivariant for the action \eqref{action_ZN-d} and for the translation in the last factor. Hence, we get the following commutative diagram 

\begin{equation} \label{diag_triv_2}
\begin{tikzcd}[ampersand replacement=\&]
	{([\C^I\times\T^J)\times\C^{N-d}/E(\ker(h_{\sigma\C}))\times \Z^{N-d}] } \&\& {\Uscr^{cal}_\sigma} \\
	{[\C^I/E(\ker(h_{\sigma\C}))] \times\T^J\times\T^{N-d}} \\
	{}
	\arrow["\pscr", from=1-1, to=1-3]
	\arrow["{(\widehat{\varphi},id)^{-1}}"', from=1-1, to=2-1]
	\arrow["{\varphi^{cal}}"', from=2-1, to=1-3]
\end{tikzcd}
\end{equation}

The diagrams \eqref{diag_triv_1} and \eqref{diag_triv_2} can be seen as trivialization of the morphism $\varphi^{cal}$. Thus, we can see $\varphi^{cal}$ as a $\C^{N-d}$-fibre bundle.

We can remark that the transition map of $\Xscr_{\Delta,\Gamma}$ comes from descent of transitions map between the different $[\C^I/E(\ker(h_{\sigma\C}))]\times \T^J$ thus we can define $\Xscr_\Delta$ the stack over $\Afrak$ given by the descent data of the family of stacks $[\C^I/E(\ker(h_{\sigma\C}))] \times \T^J \times \T^{N-d}$ indexed by every maximal cones of $\Delta$ with these transition maps (in the same manner as \ref{QTV_cal}).

By functoriality, we get 

\begin{Thm}
The map $\varphi^{cal} \colon \Xscr_{\Delta} \to \Xscr^{cal}_{\Delta,h,\Ical}$ is a $\C^{N-d}$-principal bundle (in the sense that this morphism can be locally trivialized on the affine pieces, the trivialization map induced the identity on the factors $\C^{N-d}$ and the transition maps are given by the action of $\C^{N-d}$)
\end{Thm}

This statement gives further details on the theorem 6.21 of \citep{VQS}.
\begin{proof}
It remains us to explain to the action induced by the transition maps. By the definitions of $\Xscr_\Delta$ and $\Xscr^{cal}_{\Delta,h,\Ical}$ (thanks to the decomposition $\C^I=\C^{\widetilde{I}} \oplus \ker(h_{\sigma\C})$ and the definition of the actions), we can restrict us to the simplicial fans i.e. $\Xscr_\Delta=\Sscr$. 

Each cone $\sigma_I$ of $\Delta$ induces an action of the group $\C^{N-d}$ on each variety $\C^{I} \times \T^{J} \times \C^{(I \cup J)^c} \subset \Sscr$. Namely, it is the induced $\C^{N-d}$-action by the action by translation on the last factor of $\C^{I} \times \T^{J} \times \C^{N-d}$ and the isomorphism  $(\widehat{\varphi}_I,id)  \circ (id,P_{\chi_I|\C^{N-d}})^{-1}$ i.e. 
\begin{equation} \label{action_transp}
 \lambda \cdot (z_1,z_2,w)=(E(-\widetilde{\varphi}_I^1(\lambda))z_1,E(-\widetilde{\varphi}_I^2(\lambda))z_2,w+P_{\chi_I}\lambda)   
\end{equation}
where $\widetilde{\varphi}_I=(\widetilde{\varphi}^1_I,\widetilde{\varphi}_I^2): \C^{N-d} \to \C^I \times \C^{J}$ is the restriction of $\varphi_I$ on $\C^{N-d}$.

Thus, the morphism $\widehat{\varphi_I}\circ (id,P_{\chi_I|\C^{N-d}})^{-1} \colon \C^{I} \times \T^{J} \times \C^{(I \cup J)^c} \to \C^{I} \times \T^{J}$ is a $\C^{N-d}$ trivializable principal bundle. Hence, it has a global cross section $s_\sigma:  \C^{I} \times \T^{J}\to\C^{I} \times \T^{J} \times \C^{(I \cup J)^c} $ which is defined by 
\[
s_\sigma(x,y)=(x,y,0)
\]
This morphism descends as a morphism $\Uscr_\sigma^{cal} \to \C^I \times \T^{J} \times \T^{N-d}$. 

In order to conclude, we have to find, for each cones $\sigma=\sigma_I$, $\tau=\sigma_{I'}$ with a non-empty intersection, a morphism $t_{\sigma\tau} \colon \Uscr_{\sigma\tau}^{cal} \to \C^{N-d}$ such that 
\begin{equation} \label{transition}
   s_\sigma=t_{\sigma\tau} \cdot s_\tau
\end{equation}

Note $\varphi_I=\psi_I^{-1} h P_{\chi_I}$ and $\varphi_{I'}=\psi_{I'}^{-1} h P_{\chi_{I'}}$ the calibration associated, respectively, to $\sigma$ and $\tau$.
Recall the commutative diagram used in the definition of quantum toric varieties: 
\[\begin{tikzcd}[ampersand replacement=\&]
	{\C^{I \cup J}} \& {\C^d} \& {\C^{I' \cup J'}} \\
	{\psi_I^{-1}(\Gamma)} \& \Gamma \& {\psi_{I'}^{-1}(\Gamma)} \\
	{\Z^N} \& {\Z^N} \& {\Z^N}
	\arrow["{P_{\chi_I}}", from=3-1, to=3-2]
	\arrow["{P^{-1}_{\chi_{I'}}}", from=3-2, to=3-3]
	\arrow["{\varphi_I}", from=3-1, to=2-1]
	\arrow["h"', from=3-2, to=2-2]
	\arrow["{\varphi_{I'}}"', from=3-3, to=2-3]
	\arrow["{\psi_I}", from=2-1, to=2-2]
	\arrow["{\psi_{I'}^{-1}}", from=2-2, to=2-3]
	\arrow[hook, from=2-3, to=1-3]
	\arrow[hook, from=2-2, to=1-2]
	\arrow[hook, from=2-1, to=1-1]
	\arrow["{\psi_I}", from=1-1, to=1-2]
	\arrow["{\psi_{I'}^{-1}}", from=1-2, to=1-3]
\end{tikzcd}\]
Since $\psi_{I'}^{-1} \psi_{I} \varphi_I=\varphi_{I'}(P_{\chi_I'}^{-1}P_{\chi_I}) $ then for all point $m$ of $\Z^d$, 
\begin{equation} \label{egalite_Z^d}
   \psi_{I'}^{-1} \psi_{I}(P_\chi m)=\varphi_{I'}(P_{\chi_{I'}}^{-1}P_{\chi_I})(m)=(Id+\widetilde{\varphi}_{I'}) \circ (P_{\chi_{I'}}^{-1}P_{\chi_I})(m) 
\end{equation}
The transition map between $\Uscr_{\tau\sigma}^{cal}$ and $\Uscr^{cal}_{\sigma\tau}$ is $[z^{P_\chi}] \in \Uscr_{\tau\sigma}^{cal} \mapsto \left[z^{\psi_{I'}^{-1} \psi_{I}P_\chi}\right] \in \Uscr_{\sigma\tau}^{cal}$ (i.e. the descent to quotient of the linear morphism $\psi_{I'}^{-1} \psi_{I}$)

Note $K_{II'}$ the set $P_{\chi_{I'}}^{-1}(I) \setminus \{1,\ldots,d\}$. Then, the maps $t_{\sigma\tau} \colon z \mapsto \left[z^{K_{II'}}\right]$ verify the equality \eqref{transition} thanks to the equalities \eqref{action_transp} and \eqref{egalite_Z^d}.
\end{proof}

As $\Xscr^{cal}_{\Delta,h,\Ical}$ is a gerbe over $\Xscr_{\Delta,\Gamma}$ with band $\Z^{\mathrm{rk}(\Xi)}$ then 

\begin{Thm}
The map $\mathrm{f} \circ \varphi^{cal} \colon \Xscr_{\Delta} \to \Xscr_{\Delta,\Gamma}$ is a $\C^{N-d-\mathrm{rk}(\Xi)}\times \T^{\mathrm{rk}(\Xi)}$-principal bundle and hence, a $\T^{\mathrm{rk}(\Xi)}$-principal bundle up to homotopy.
\end{Thm}

\subsubsection*{Acknowledgements}
I want to thank my Ph.D. advisor Laurent Meersseman for the numerous discussions and for the corrections of the first versions of this paper, Ernesto Lupercio for his comments on it and Elisa Prato for pointing out the work \citep{battaglia2008geometric} of Fiammetta Battaglia.

\bibliographystyle{abbrv}
\bibliography{biblio} 

\begin{thebibliography}{10}

\bibitem{battaglia2008geometric}
F.~Battaglia.
\newblock Geometric spaces from arbitrary convex polytopes.
\newblock {\em International Journal of Mathematics}, 23(01):1250013, 2012.

\bibitem{Battaglia_2015}
F.~Battaglia and D.~Zaffran.
\newblock Foliations modeling nonrational simplicial toric varieties.
\newblock {\em International Mathematics Research Notices}, Feb 2015.

\bibitem{intro_stack}
K.~Behrend, B.~Conrad, D.Edidin, B.~Fantechi, W.Fulton, L.~Göttsche, and
  A.~Kresch.
\newblock Introduction to stacks.
\newblock
  \url{https://www.math.uzh.ch/index.php?id=ve_vo_det&key2=580&semId=13}.

\bibitem{bosio2006}
F.~Bosio and L.~Meersseman.
\newblock Real quadrics in {$\C^n$} , complex manifolds and convex polytopes.
\newblock {\em Acta Math.}, 197(1):53--127, 2006.

\bibitem{Coxquotient}
D.~Cox.
\newblock The homogeneous coordinate ring of a toric variety, revised version.
\newblock 11 1992.

\bibitem{cox}
D.~Cox, J.~Little, and H.~Schenck.
\newblock {\em Toric Varieties}.
\newblock Graduate studies in mathematics. American Mathematical Soc., 2011.

\bibitem{Stackeveryb}
B.~Fantechi.
\newblock {\em Stacks for Everybody}.
\newblock Birkh{\"a}user Basel, Basel, 2001.

\bibitem{fulton}
W.~Fulton.
\newblock {\em {Introduction to toric varieties}}.
\newblock Annals of mathematics studies. Princeton Univ. Press, Princeton, NJ,
  1993.

\bibitem{CP}
B.~Gr{\"u}nbaum, V.~Kaibel, V.~Klee, and G.~Ziegler.
\newblock {\em Convex Polytopes}.
\newblock Graduate Texts in Mathematics. Springer, 2003.

\bibitem{hoffman2020toric}
B.~Hoffman.
\newblock Toric symplectic stacks.
\newblock {\em Advances in Mathematics}, 368:107135, 2020.

\bibitem{toricfoliation}
H.~Ishida.
\newblock Torus invariant transverse {K}\"{a}hler foliations.
\newblock {\em Transactions of the American Mathematical Society}, 369, 05
  2015.

\bibitem{jiang2008}
Y.~Jiang.
\newblock The orbifold cohomology ring of simplicial toric stack bundles.
\newblock {\em Illinois J. Math.}, 52(2):493--514, 2008.

\bibitem{katzarkov:hal-01672716}
L.~Katzarkov, E.~Lupercio, L.~Meersseman, and A.~Verjovsky.
\newblock {The definition of a non-commutative toric variety}.
\newblock {\em {Contemporary mathematics}}, 620:223--250, 2014.

\bibitem{VQS}
L.~Katzarkov, E.~Lupercio, L.~Meersseman, and A.~Verjovsky.
\newblock Quantum (non-commutative) toric geometry: Foundations.
\newblock {\em arXiv preprint math/2002.03876}, 2020.

\bibitem{LV}
L.~Meersseman and A.~Verjovsky.
\newblock Holomorphic principal bundles over projective toric varieties.
\newblock 2004(572):57--96, 2004.

\bibitem{quasifold}
E.~Prato.
\newblock Simple non-rational convex polytopes via symplectic geometry.
\newblock {\em Topology}, 40:961--975, 2001.

\bibitem{senechal1996quasicrystals}
M.~Senechal.
\newblock {\em Quasicrystals and Geometry}.
\newblock Cambridge University Press, 1996.

\bibitem{stacks-project}
T.~{Stacks Project Authors}.
\newblock \textit{Stacks Project}.
\newblock \url{https://stacks.math.columbia.edu}, 2018.

\bibitem{Morita}
P.~Xu.
\newblock Momentum maps and {M}orita equivalence.
\newblock {\em Journal of Differential Geometry}, 67:289--333, 2004.

\end{thebibliography}

\end{document}